\documentclass[fleqn]{elsarticle}
\usepackage{amsmath,amsfonts,graphicx,subfigure,epstopdf,float,amssymb,bm,amsthm,multirow,appendix,bbm,color,enumitem}

\usepackage{geometry}
\usepackage{geometry}
\usepackage[colorlinks,linkcolor=red,citecolor=blue]{hyperref}
\usepackage{booktabs,multirow}
\numberwithin{equation}{section}
\newtheorem{thm}{Theorem}[section]

\newtheorem{rmk}{Remark}[section]
\newtheorem{scm}{Scheme}[section]
\newtheorem{lem}{Lemma}[section]

\graphicspath{{FIGS/}}

\def\sech{\mathrm{sech}}

\allowdisplaybreaks[4] 

\begin{document}

\begin{frontmatter}

\title{Arbitrary high-order structure-preserving methods for the quantum Zakharov system}

\author[YNU]{Gengen Zhang}
\ead{zhanggen036@163.com}
\author[YNUFE]{Chaolong Jiang\corref{cor1}}
\cortext[cor1]{Corresponding author}
\ead{Chaolong$\_$jiang@126.com}
\address[YNU]{School of Mathematics and Statistics, Yunnan University, Kunming 650504,  China}
\address[YNUFE]{School of Statistics and Mathematics,
Yunnan University of Finance and Economics, Kunming 650221, China}
	\begin{abstract}
In this paper, we present a new methodology to develop arbitrary high-order structure-preserving methods for solving the quantum Zakharov system. The key ingredients of our method are: (i) the original Hamiltonian energy is reformulated into a quadratic form by introducing a new quadratic auxiliary variable; (ii) based on the energy variational principle, the original system is then rewritten into a new equivalent system which inherits the mass conservation law and a quadratic energy; (iii) the resulting system is discretized by symplectic Runge-Kutta method in time combining with the Fourier pseudo-spectral method in space. The proposed method achieves arbitrary high-order accurate in time and can preserve the discrete mass and original Hamiltonian energy exactly. Moreover, an efficient iterative solver is presented to solve the resulting discrete nonlinear equations. Finally, ample numerical examples are presented to demonstrate the theoretical claims and illustrate the efficiency of our methods.
	\end{abstract}
	
	
\end{frontmatter}

\section{Introduction}

The quantum Zakharov system \cite{Garcia05,Haas09} is widely used to describe
 the nonlinear interaction between the Langmuir waves and the ion-acoustic waves. In this paper, we consider the following quantum Zakharov system (QZS):
\begin{equation}\label{eq1.1}
\left\{
\begin{split}
& {\rm i} \partial_tE({\bf x},t)+\Delta E({\bf x},t)- \varepsilon^2 \Delta^2 E({\bf x},t)=N({\bf x},t) E({\bf x},t), \\
&\partial_{tt}N({\bf x},t)-\Delta N({\bf x},t)+\varepsilon^2 \Delta^2 N({\bf x},t) =\Delta|E({\bf x},t)|^2, \ {\bf x} \in \Omega,\ 0<t\le T,\\
& E ({\bf x},0) = E_0({\bf x}), \  N ({\bf x},0) = N_0({\bf x}),  \ N_t({\bf x},0) = N_1({\bf x}),\ {\bf x}\in\bar{\Omega}\subset\mathbb{R}^d,\ d=1,2,
 \end{split}\right.
\end{equation}
where $\text{i}=\sqrt{-1}$ is the complex unit, $t$ is the time variable, ${\bf x}$ is the spatial variable, the complex-valued function $E:=E({\bf x},t)$ denotes the slowly varying envelope of the rapidly oscillatory electric field, the real-valued function $N:=N({\bf x},t)$ represents the low-frequency variation of the density of the ions, $\Delta$ is the usual Laplace operator, the quantum effect $\varepsilon >0 $ is the ratio of the ion plasma
and the temperature of electrons, and $E_{0}(\mathbf{x})$, $N_{0}(\mathbf{x})$ and $N_{1}(\mathbf{x})$ are given initial conditions, and $N_{1}(\mathbf{x})$ satisfies the following compatibility condition \cite{CaiYongyong22ANM,Glassey92}:
\begin{equation}\label{N1compatibility}
\int_{\Omega} N_{1}(\mathbf{x}) d \mathbf{x} =0.
\end{equation}
In the special case $\varepsilon=0$, it reduces
to the classical Zakharov system (ZS), which
have been widely applied to various physical problems, such as plasmas \cite{GuoBoling16,Zakharov72},
hydrodynamics \cite{Degtyarev74}, the
theory of molecular chains \cite{Davydov79} and so on.
 When either the electrons temperature is low or the ion-plasma frequency is high, the quantum effect  can be characterized by the fourth-order perturbation with a quantum parameter $\varepsilon$. For more details, please refer to Refs. \cite{Haas11,Marklund05,Misra09}.
With the periodic boundary condition, the QZS \eqref{eq1.1} conserves the mass
\begin{equation}\label{eq1.2-mass}
\begin{split}
& \mathcal{M}(t)=\int_{\Omega} |E|^2 d{\bf x}\equiv\mathcal{M}(0),\ t\ge 0,
 \end{split}
\end{equation}
and the energy
\begin{equation}\label{eq1.2-energy}
\begin{split}
&\mathcal{H}(t)=\int_{\Omega} \left( \left|\nabla  E \right|^{2} +\frac{1}{2}( |\nabla v |^{2} + N^{2}) + \varepsilon^{2}\left| \triangle E \right|^{2}
 +\frac{\varepsilon^{2}}{2}\left|\nabla N  \right|^{2} + N|E|^{2} \right) d{\bf x}\equiv\mathcal{H}(0), \ t\ge 0,
 \end{split}
\end{equation}
where $\triangle v  = N_{t}$.

Extensive mathematical and numerical studies have been carried out for the
above QZS \eqref{eq1.1} in the literature. Along the mathematical front,
Fang et al. \cite{FangYungFu19} showed that the QZS \eqref{eq1.1} is locally well-posed in $L^2(\mathbb{R}^d)$ data
for dimension up to eight, together with global existence for dimensions up to five,
which is different from the classical ZS where the global and local well-posedness of the Cauchy problem is known only for $1\le d\le 3$ \cite{FangYungFu16,Glang}. The blow-up in finite time of the solution for high-dimensional classical ZS was investigated in \cite{Masselin01}.
Along the numerical front, the numerical studies of  the classical or generalized ZS are very rich, such as
time splitting  methods \cite{BaoSun05, BaoSu18, JinS04,JinS06}, scalar auxiliary variable  approach  \cite{Shen2021ZS}, finite difference methods \cite{BaoSu17,CaiYY18, Glassey92,XiaoAiguo19ANM}, discontinuous Galerkin method \cite{XiaShu10}, etc. Recently, there has been growing interest in developing accurate and efficient numerical methods for the QZS \eqref{eq1.1}. Baumstark and Schratz \cite{Baumstark21} developed a new class of asymptotic preserving trigonometric integrators for the QZS \eqref{eq1.1}. Meanwhile, it is shown rigorously in mathematics that the scheme converges to the classical ZS in the limit $\varepsilon\rightarrow 0 $ uniformly in the time discretization parameter. Zhang \cite{ZhangG21} proposed an explicit mass-conserving time-splitting exponential wave integrator Fourier pseudo-spectral (TS-EWI-FP) method for the QZS \eqref{eq1.1}. However, these mentioned schemes cannot conserve the energy \eqref{eq1.2-energy} of the QZS \eqref{eq1.1}. It has been shown that in the numerical simulation of the collision of solitons, the solution of mass- and energy-conserving schemes cannot produce nonlinear blow-up \cite{SV1986IMA,ZFPV95}. Later on, Zhang and Su \cite{ZhangSu21} proposed and analyzed a linearly-implicit conservative compact
finite difference method for the QZS \eqref{eq1.1}. Cai et al. \cite{CaiYongyong22ANM}  proposed a novel of mass- and energy-conserving compact finite difference scheme for the
 QZS \eqref{eq1.1}. However it is shown rigorously in mathematics that these existing mass- and energy-conserving schemes are only second-order accurate in time. In \cite{CWJ2021cpc,GZJsiam2020,JCQSjsc2022}, it is clear to observe that high-order accurate structure-preserving scheme will provide much smaller numerical error and more robust than the second-order accurate one as the large time step is chosen. Thus, it is desirable to propose high-order accurate mass- and energy-preserving methods for solving the QZS \eqref{eq1.1}.

 As a matter of the fact, in the past few decades, how to devise high-order accurate energy-preserving schemes for conservative systems have attracted plenty of attention. The excellent ones include Hamiltonian Boundary Value Methods (HBVMs) \cite{BI16,BIT10}, energy-preserving variant of collocation methods \cite{CH11bit,H10}, high-order averaged vector field (AVF) methods \cite{LwQ14,QM08,MHW2021jcp}, functionally
fitted energy-preserving methods \cite{MB16,LWsina16,WW18} and energy-preserving continuous stage Runge-Kutta (RK) methods \cite{MB16,TS12}, etc.  These methods can be easily extended to propose high-order accurate energy-preserving scheme for the QZS \eqref{eq1.1}, which however cannot preserve the mass exactly \cite{BBCIamc18,LW15}. Based on the basic principle of the structure-preserving method where the numerical method should preserve the intrinsic properties of the original problems as much as possible, it is valuable to expect that the high-order mass and energy-preserving discretizations for the QZS \eqref{eq1.1} will produce richer information on the continuous system.
Very recently, inspired by the ideas of the invariant energy quadratization (IEQ) approach \cite{YZW17}, a new class of high-order accurate energy-preserving
methods are proposed in \cite{GongQuezheng21,GHWW2022}, which was generalized more recently by Tapley\cite{Tapley-SISC2022}. Especially, the term ``quadratic auxiliary variable (QAV)  approach" was coined by Gong et al. in \cite{GongQuezheng21}. The key differences between the IEQ approach and the QAV approach are: (i) the auxiliary variable introduced by the QAV approach shall be quadratic; (ii) as a high-order quadratic invariant-preserving method in time is applied to the equivalent system, the resulting method can not only preserve the original Hamiltonian energy instead of a modified energy \cite{FCW2021anm,JCW19jsc,JWG2020anm}, but also the original quadratic invariants of the equivalent system. Thus, the QAV approach will be an efficient strategy to develop high-order accurate mass- and energy-preserving scheme for the QZS \eqref{eq1.1}, however to our knowledge, there has been no reference considering this issue.

In this paper, motivated by the QAV approach, we first introduce a quadratic auxiliary variable to transform the Hamiltonian energy into a quadratic from, and the original system is then reformulated into a new equivalent system. Finally, a fully-discrete method is obtained by using the RK method in time and the Fourier pseudo-spectral method in space for the reformulated system. We show that when the symplectic RK method is selected, the proposed method can conserve the discrete mass and original Hamiltonian energy exactly. In addition, to solve the discrete nonlinear equations of our method efficiently, an efficient fixed-pointed iteration method is proposed.

The rest of the paper is organized as follows.
In section 2, we first reformulate the original QZS \eqref{eq1.1} into an equivalent form, and the energy conservation law and the mass conservation law of the reformulated system are then investigated. In section 3, a new class of high-order structure-preserving schemes are proposed based on the symplectic RK  method in time and the Fourier pseudo-spectral method in space, respectively.
In Section 4, an efficient implementention for the proposed scheme  is presented.
In Section 5, extensive numerical experiments for the QZS \eqref{eq1.1} are carried out to illustrate
the capability and accuracy of the method, and show some complex dynamical behaviors.
 Finally, some conclusions are drawn in Section 6.

\section{Model reformulation}

Denote $u=(E,N,v)^T $, the QZS \eqref{eq1.1} can be written into the following energy-conserving system
\begin{align}
\begin{split}\label{QSZ-CES}
\partial_t u= S \frac{\delta \mathcal{H}}{\delta \bar{u}},
 \end{split}
 \end{align}
where $\bar{u}$ is the complex conjugate of $u$,
\begin{align*}S=\left(\begin{array}{ccc}
{\rm i} &  0 & 0 \\
0 & 0& 1 \\
0& -1 & 0
\end{array} \right)
\end{align*} is the skew-adjoint operator, and the Hamiltonian energy functional
\begin{equation}\label{Hamilton-energy}
\mathcal{H}[u] = \int_{\Omega} \left( -\left| \nabla E\right|^{2} -\frac{1}{2}\left( |\nabla v|^{2} + N^{2}\right) - \varepsilon^{2}\left| \Delta E\right|^{2}
 -\frac{\varepsilon^{2}}{2}\left|\nabla N\right|^{2} - N|E|^{2} \right) d {\bf x}.
\end{equation}

Next, based on the idea of the QAV approach, we first introduce a quadratic auxiliary variable, as follows:
\begin{equation}\label{auxiliary-variable}
	q:= q({\bf x},t)=|E|^{2},
\end{equation}
the original energy \eqref{Hamilton-energy} is then transformed into the following quadratic form
\begin{equation}\label{modified-energy}
	\mathcal{E}[u,q] =  \int_{\Omega} \left( -\left| \nabla E\right|^{2} -\frac{1}{2}\left( |\nabla v|^{2} + N^{2}\right) - \varepsilon^{2}\left| \Delta E\right|^{2}
 -\frac{\varepsilon^{2}}{2}\left|\nabla N\right|^{2} - Nq \right) d {\bf x}.
\end{equation}

According to the energy variational principle, we rewritten the QZS \eqref{QSZ-CES} into a new equivalent system
\begin{equation}\label{NEWsystem}
	\begin{cases}
 E_t  = {\rm i} ( \Delta E- \varepsilon^2 \Delta^2 E -N E  )  , \\
 N_t = \Delta v,  \\
 v_{t} =   N - \varepsilon^2 \Delta  N + q,  \\
 q_t = 2 {\rm Re} (E_t\cdot \bar{E}),\\
	\end{cases}
\end{equation}
with the consistent initial conditions
\begin{align}\label{CIC}
	&E ({\bf x},0) = E_0({\bf x}), \  N ({\bf x},0) = N_0({\bf x}), \ q({\bf x},0) = \big|E_0({\bf x})\big|^2,
\ N_t({\bf x},0) = N_1({\bf x}), \nonumber\\
& \Delta v({\bf x},0)= N_1({\bf x}),\ \int_{\Omega} v(\mathbf{x},0) d \mathbf{x} =0,\ {\bf x}\in\bar{\Omega},
\end{align}
where ${\rm Re}(\bullet)$ represents the real part of $\bullet$, and we impose $\int_{\Omega} v(\mathbf{x},0) d \mathbf{x} =0$
in the equivalent form  so that $v(\mathbf{x},t)$ is uniquely defined in the second equality of \eqref{NEWsystem}\cite{CaiYongyong22ANM}.

Subsequently, we focus on investigating the structure-preserving properties of the reformulated system \eqref{NEWsystem}-\eqref{CIC}.
\begin{thm}
	Under the periodic boundary conditions, the system \eqref{NEWsystem}-\eqref{CIC} possesses the following conservation laws:
	\begin{itemize}
\item The mass
\begin{align}\label{QAVMCL}
		& \mathcal M(t) \equiv \mathcal M(0),\ t\ge 0,
\end{align}
where $\mathcal{M}(t)$ is defined by \eqref{eq1.2-mass}.
\item The two quadratic invariants
\begin{align}\label{QAVCL}
		& q({\bf x},t)-|E({\bf x},t)|^{2} = q({\bf x},0)-|E({\bf x},0)|^{2}\equiv0,\ {\bf x}\in \Omega,\; t\ge 0, \\\label{QAVECL}
		& \mathcal{E}(t) \equiv \mathcal{E}(0),\ t\ge 0,
	\end{align}
where $\mathcal{E}(t)$  are defined by \eqref{modified-energy}.
\end{itemize}
\end{thm}

\begin{proof} We make the inner product of the first equation of \eqref{NEWsystem} with $E$ and then take the real part of the
resulting equation to obtain
\begin{equation}\label{MCL1}
	\frac{\mathrm{d}}{\mathrm{d}t} (|E|^2,1) 
= 2 {\rm Re} \int_{\Omega} {\rm i} \left( -\left| \nabla E\right|^{2} - \varepsilon^{2}\left| \Delta E\right|^{2}
 - N|E|^{2} \right) d{\bf x}  = 0,
\end{equation}
which implies that the system \eqref{NEWsystem} satisfies \eqref{QAVMCL}.

Combining the initial condition $q({\bf x},0)=|E({\bf x},0)|^2$ with the fourth equation of the system \eqref{NEWsystem}, we can deduce
	\begin{equation*}
		\partial_t (q - |E({\bf x},t)|^{2}) = 0,
	\end{equation*}
	which yields \eqref{QAVCL}.
	
Using the periodic boundary condition and the system \eqref{NEWsystem}, we then obtain
	\begin{align*}
		\frac{\mathrm{d}}{\mathrm{d}t}\mathcal{E}(t)
		&= 2{\rm Re}( \Delta E, E_t)-2\varepsilon^{2}{\rm Re} ( \Delta^2 E, E_t)-2{\rm Re}( NE, E_t)
    - (N,N_t)    + \varepsilon^{2} (\Delta N,N_t) -(q,N_t) + (v_t,N_t)  \\
		&= 2{\rm Re} \left( \Delta E - \varepsilon^{2}\Delta^2 E - NE, E_t  \right)
          +  \left(-N + \varepsilon^{2} \Delta N -q + v_t,N_t \right) \\
		&= 2{\rm Re} \left( -{\rm i} E_t, E_t  \right)\\
		&= 0,
	\end{align*}
where $(f,g) = \int_{\Omega} f \bar{g} \mathrm{d} {\bf x}$, for any $ f, g\in L^2(\Omega)$ is denoted as the continuous $L^2$ inner product.
This completes the proof.
\end{proof}

\section{High-order accurate mass- and energy-preserving methods}

In this section, the symplectic RK methods are first used to discretize the  system \eqref{NEWsystem} in time and a class of semi-discrete
RK method is proposed, which can conserve the mass and original Hamiltonian energy of the QZS \eqref{eq1.1}. Subsequently, the Fourier pseudo-spectral method is then employed to discretize the spatial variables of the semi-discrete scheme and a class of fully-discrete high-order mass- and energy-preserving scheme are presented.

\subsection{Time semi-discretization}

Let the time step $\tau=\frac{T}{J}$ and denote $t_n = n \tau$ for $0 \leq n \leq J$.
Let $w^n$ and $w_{ni}$ be the numerical approximations of the function $w({\bf x}, t)$ at $t_n$ and $t_n+c_i\tau$, respectively.
Applying an $s$-stage RK method to discrete the system \eqref{NEWsystem} in time and the following semi-discrete scheme is presented.
\begin{scm} \label{QZS-Scheme-3.1}
Let $b_i,a_{ij}({i,j=1,\cdots,s})$ be real numbers and let $c_i=\sum_{j=1}^sa_{ij}$.
 For given $(E^{n},N^{n}, v^{n}, q^n)$, an $s$-stage RK method is given by
	\begin{equation}\label{QZS-Scheme-eq-3.1}
		\begin{cases}
			E_{ni} =  E^n +  \tau \sum\limits_{j=1}^s a_{ij} k_j^{1},\ k_i^{1} = {\rm i} ( \Delta E_{ni}- \varepsilon^2 \Delta^2 E_{ni} -N_{ni} E_{ni} ), \\
	    	N_{ni} =  N^n +  \tau \sum\limits_{j=1}^s a_{ij} k_j^{2},\ k_i^{2}=  \Delta v_{ni}, \\
			v_{ni} = v^n +  \tau \sum\limits_{j=1}^s a_{ij}  k_j^{3},\ k_i^{3}=   N_{ni} - \varepsilon^2 \Delta  N_{ni} + Q_{ni}, \\
			Q_{ni} = q^n +  \tau \sum\limits_{j=1}^s a_{ij}  k_j^{4},\ k_i^{4}=   2 {\rm Re} ( \bar{E}_{ni}\cdot k_i^{1} ).
		\end{cases}
	\end{equation}
	Then $(E^{n+1},N^{n+1}, v^{n+1}, q^{n+1})$ is updated by
	\begin{align}
		E^{n+1} &=  E^n +  \tau \sum\limits_{i=1}^s b_{i} k_i^{1},\label{TD-E} \\
		N^{n+1}& =  N^n +  \tau \sum\limits_{i=1}^s b_{i} k_i^{2}, \label{TD-N}\\
		v^{n+1} &=  v^n +  \tau \sum\limits_{i=1}^s b_{i} k_i^{3},\label{TD-v} \\
		q^{n+1} &=  q^n +  \tau \sum\limits_{i=1}^s b_{i} k_i^{4}. \label{TD-q}
	\end{align}
\end{scm}

\begin{thm}\label{thm:SD-CL}
If the coefficients of the RK method \eqref{QZS-Scheme-3.1} satisfy
	\begin{equation}\label{RK-symplectic-condition}
		b_i a_{ij} + b_ja_{ji} = b_ib_j,\quad \forall~i,j = 1,\cdots,s,
	\end{equation}
and the periodic boundary condition is considered, the {\bf Scheme \ref{QZS-Scheme-3.1}} preserves the following semi-discrete conservation laws
\begin{itemize}
\item The semi-discrete mass
	\begin{align}\label{QAVRKMCL}
		& \mathcal{M}^{n+1}=\mathcal{M}^n,\ \mathcal{M}^n=\big(|E^{n }|^2,1\big),\ n=0,1,2,\cdots, J.
\end{align}
\item The two semi-discrete quadratic invariants
\begin{align}\label{QAVRKCL}
		&  q^n-|E^{n}|^2=0, \\ \label{QAVRKECL}
		& \mathcal{E}^{n+1} = \mathcal{E}^{n},\ n=0,1,2,\cdots, J,
	\end{align}
	where
\begin{align*}
&\mathcal{E}^n = (\Delta E^n, E^n)-\varepsilon^{2} (\Delta E^n,\Delta E^n) -\frac{1}{2}(N^n,N^n)
     + \frac{\varepsilon^{2}}{2}(\Delta N^n,N^n) -(N^n, q^n) + \frac{1}{2} (\Delta v^n, v^n).
\end{align*}
\end{itemize}
\end{thm}

\begin{proof}
	According to Eq. \eqref{TD-E}, we have
 \begin{equation}  \label{u2diff1}
 \begin{array}{lll}
\mathcal{M}^{n+1} = \mathcal{M}^{n} + \tau \sum\limits_{i=1}^sb_{i} \big(k_i^{1},E^{n}\big) +\tau \sum\limits_{i=1}^sb_{i} \big(E^{n},k_i^{1}\big) +\tau^2 \sum\limits_{i,j=1}^s b_{i}b_{j}\big( k_i^{1},k_j^{1}\big).
	\end{array}
\end{equation}
Plugging $E^n = E_{ni} -\tau \sum\limits_{j=1}^s a_{ij} k_j^{1}$ into the right hand side of \eqref{u2diff1} and using \eqref{RK-symplectic-condition}, we have
	\begin{equation}\label{u2diff2}
	\mathcal{M}^{n+1} = \mathcal{M}^{n} +\tau \sum\limits_{i=1}^s b_{i} \big(E_{ni},k_i^{1}\big)
              + \tau \sum\limits_{i=1}^s b_{i} \big(k_i^{1},E_{ni}\big).
	\end{equation}
Furthermore, we can deduce
 \begin{align*}
\tau \sum\limits_{i=1}^s b_{i} \big(E_{ni},k_i^{1}\big) &+ \tau \sum\limits_{i=1}^s b_{i} \big(k_i^{1},E_{ni}\big) \\
&= 2\tau \sum\limits_{i=1}^s b_{i} {\rm Re} (k_i^1, \bar{E}_{ni}) \\
&= 2\tau \sum\limits_{i=1}^s b_{i} {\rm Re} \Big({\rm i} ( \Delta E_{ni}- \varepsilon^2 \Delta^2 E_{ni} -N_{ni} E_{ni} ), \bar{E}_{ni}\Big) \\
&= 2\tau \sum\limits_{i=1}^s b_{i} {\rm Re} \Big({\rm i} \big( -|\nabla E_{ni}|^2 - \varepsilon^2 |\Delta  E_{ni}|^2 -N_{ni} |E_{ni}|^2 \big),1\Big) \\
&=0.
\end{align*}
Then, combining the above equation with \eqref{u2diff2}, we obtain the discrete mass  conservation law  \eqref{QAVRKMCL}.

Based on \eqref{TD-E}, \eqref{RK-symplectic-condition} and $E^n = E_{ni} -\tau \sum\limits_{j=1}^s a_{ij} k_j^{1}$, we get
\begin{align}
 |E^{n+1}|^2 - |E^{n }|^2 &= E^{n+1} \cdot \bar{E}^{n+1}  - E^{n}\cdot \bar{E}^{n}  \nonumber \\
 & =  \tau \sum\limits_{i=1}^sb_{i}  ( k_i^{1} \cdot \bar{E}^{n} )
    +\tau \sum\limits_{i=1}^sb_{i}  ( \overline{k}_i^{1} \cdot E^{n} )
   +\tau^2 \sum\limits_{i,j=1}^s b_{i}b_{j} ( k_i^{1} \cdot \bar{k}_j^{1} ) \nonumber \\
 & =  \tau \sum\limits_{i=1}^sb_{i}  ( k_i^{1} \cdot \bar{E}_{ni} )
    +\tau \sum\limits_{i=1}^sb_{i}  ( \bar{k}_i^{1} \cdot E_{ni} )
   +\tau^2 \sum\limits_{i,j=1}^s (- b_i a_{i j} - b_j a_{j i}+ b_{i}b_{j}  ) ( k_i^{1} \cdot \bar{k}_j^{1} ) \nonumber \\
  & =  \tau \sum\limits_{i=1}^sb_{i}  ( k_i^{1} \cdot \bar{E}_{ni} )
    +\tau \sum\limits_{i=1}^sb_{i}  ( \bar{k}_i^{1} \cdot E_{ni} )      .  \label{u2diff3a}
\end{align}
Noticing that
	\begin{equation}\label{u2diff3}
    \begin{array}{lll}
	 q^{n+1} - q^n   =  \tau \sum\limits_{i=1}^s b_i k_i^4   =  2\tau \sum\limits_{i=1}^s b_i {\rm Re}(k_i^1 \cdot \bar{E}_{ni})
= \tau \sum\limits_{i=1}^sb_{i}  ( k_i^{1} \cdot \bar{E}_{ni} )
    +\tau \sum\limits_{i=1}^sb_{i}  ( \bar{k}_i^{1} \cdot E_{ni} ) .
	\end{array}
    \end{equation}
It follows from \eqref{u2diff3a} and \eqref{u2diff3}  that
\begin{align}
 q^{n+1}- |E^{n+1}|^2 =  q^n - |E^{n }|^2 .
\end{align}
With the help of the initial condition $ q^0  = |E^{0}|^2 $, we can obtain \eqref{QAVRKCL}.
	
Similar to Eq. \eqref{u2diff2}, we can obtain from the \textbf{Scheme \ref{QZS-Scheme-3.1}} that
  \begin{align*}
  & (\Delta E^{n+1}, E^{n+1}) -(\Delta E^n, E^n) = 2\tau \sum\limits_{i=1}^s b_{i} {\rm Re} \big(k_i^{1},\Delta E_{ni}\big),\\
  & (\Delta E^{n+1},\Delta E^{n+1}) -(\Delta E^n,\Delta E^n)= 2\tau \sum\limits_{i=1}^s b_{i} {\rm Re} \big(k_i^{1},\Delta^2 E_{ni}\big),\\
  & (N^{n+1},N^{n+1})-(N^n,N^n)= 2\tau \sum\limits_{i=1}^s b_{i} \big(N_{ni},k_i^{2}\big),\\
   & (\Delta N^{n+1},N^{n+1})-(\Delta N^n,N^n)= 2\tau \sum\limits_{i=1}^s b_{i} \big(\Delta N_{ni},k_i^{2}\big),\\
   & (N^{n+1}, q^{n+1})-(N^n, q^n)= \tau \sum\limits_{i=1}^s b_{i} \left[ \big(k_i^{2},Q_{ni}\big) +\big(N_{ni},k_i^{4}\big)\right],\\
   & (\Delta v^{n+1}, v^{n+1})-(\Delta v^n, v^n)= 2\tau \sum\limits_{i=1}^s b_{i} \big(\Delta v_{ni},k_i^{3}\big).
	\end{align*}
Thus, together with  \eqref{QZS-Scheme-eq-3.1},  we can derive
\begin{align*}
\mathcal{E}^{n+1} - \mathcal{E}^n&= \tau \sum\limits_{i=1}^s b_i \left[ 2{\rm Re} \big(k_i^{1},\Delta E_{ni}\big)
  -2\varepsilon^2 {\rm Re} \big(k_i^{1},\Delta^2 E_{ni}\big) - \big(N_{ni},k_i^{2}\big) \right.\\
&~~~\left. +\varepsilon^2 \big(\Delta N_{ni},k_i^{2}\big) -\big(N_{ni},k_i^{4}\big)
  -\big(k_i^{2},Q_{ni}\big) +\big(\Delta v_{ni},k_i^{3}\big)   \right] \\
&= \tau \sum\limits_{i=1}^s b_i \left[ 2{\rm Re} \big(k_i^{1},\Delta E_{ni} - \varepsilon^2 \Delta^2 E_{ni}- N_{ni}E_{ni} \big)
 -\big(N_{ni}-\varepsilon^2 \Delta N_{ni}+Q_{ni}-k_i^{3}, k_i^{2}\big)   \right] \\
&= \tau \sum\limits_{i=1}^s b_i \left[ 2{\rm Re} \big(k_i^{1},-{\rm i} k_i^{1} \big)
 -\big(N_{ni}-\varepsilon^2 \Delta N_{ni}+Q_{ni}-k_i^{3}, k_i^{2}\big)   \right] \\
&= 0,
\end{align*}
	which implies that the {\bf Scheme \ref{QZS-Scheme-3.1}} satisfies \eqref{QAVRKECL}.
\end{proof}

\begin{thm}\label{thm:original-energy-conservation}
	Under the consistent initial condition \eqref{CIC}, the periodic boundary condition and the condition \eqref{RK-symplectic-condition},
the {\bf Scheme \ref{QZS-Scheme-3.1}} conserves the original Hamiltonian energy at each time step, as follows:
	\begin{equation}\label{QAVRK-OECL}
		\mathcal{H}^{n} \equiv \mathcal{H}^0, \ n=1,2,\cdots,J,
	\end{equation}
	where
\begin{align*}
\mathcal{H}^n = (\Delta E^n, E^n)-\varepsilon^{2} (\Delta E^n,\Delta E^n) -\frac{1}{2}(N^n,N^n)
     + \frac{\varepsilon^{2}}{2}(\Delta N^n,N^n) -(N^n, |E^n|^2) + \frac{1}{2} (\Delta v^n, v^n).
     \end{align*}
\end{thm}

\begin{proof} As the consistent initial condition \eqref{CIC}, the periodic boundary condition and the condition \eqref{RK-symplectic-condition} are considered, we can obtain from \textbf{Theorem \ref{thm:SD-CL}} that
	\begin{equation}\label{QZS-equation-3.16}
		q^n = |E^n|^2,\ n=0,1,2,\cdots,J,
	\end{equation}
	and
	\begin{equation}\label{QZS-equation-3.17}
		\mathcal{E}^{n+1}=\mathcal{E}^{n},\ \ n=0,1,2,\cdots,J,
	\end{equation}
where
\begin{align*}
\mathcal{E}^{n} = (\Delta E^n, E^n)-\varepsilon^{2} (\Delta E^n,\Delta E^n) -\frac{1}{2}(N^n,N^n)
     + \frac{\varepsilon^{2}}{2}(\Delta N^n,N^n) -(N^n, q^n) + \frac{1}{2} (\Delta v^n, v^n).
\end{align*}
Then substituting \eqref{QZS-equation-3.16} into \eqref{QZS-equation-3.17}, we obtain \eqref{QAVRK-OECL}. We finish the proof.
 \end{proof}


	\begin{scm} \label{scheme:SPRK} As the symplectic RK method is selected, the {\bf Scheme \ref{QZS-Scheme-3.1}} is equivalent to the following s-stage RK (using \eqref{QAVRKCL} of Theorem \ref{thm:SD-CL}):
		\begin{equation}\label{TDIV-QAV-EPRK}
			\begin{cases}
				E_{ni} =  E^n +  \tau \sum\limits_{j=1}^s a_{ij} k_j^{1}, \\
            k_i^{1} = {\rm i} ( \Delta E_{ni}- \varepsilon^2 \Delta^2 E_{ni} -N_{ni} E_{ni} ), \\
	    	N_{ni} =  N^n +  \tau \sum\limits_{j=1}^s a_{ij} k_j^{2}, \\
            k_i^{2} =  \Delta v_{ni}, \\
			v_{ni} = v^n +  \tau \sum\limits_{j=1}^s a_{ij}  k_j^{3}, \\
            k_i^{3}=   N_{ni} - \varepsilon^2 \Delta  N_{ni} + |E^n|^2 + 2 \tau \sum\limits_{j=1}^sa_{ij}  {\rm Re} ( \bar{E}_{nj}\cdot k_j^{1} ),
			\end{cases}
		\end{equation}
		where $(E^{n+1},N^{n+1}, v^{n+1})$ is obtained by
		\begin{align*}
		&	E^{n+1} =  E^n +  \tau \sum\limits_{i=1}^s b_{i} k_i^{1}, \\
		&N^{n+1} =  N^n +  \tau \sum\limits_{i=1}^s b_{i} k_i^{2}, \\
		&v^{n+1} =  v^n +  \tau \sum\limits_{i=1}^s b_{i} k_i^{3},
		\end{align*}
which implies that the QAV approach need introduce an auxiliary variable, but the auxiliary variable
can be eliminated in practical computations. Thus, it cannot increase additional computational costs.
	\end{scm}

\begin{rmk}\label{rmk-ZQs-3.1}
	We should note that the Gauss method where the RK coefficients $c_1,c_2,\cdots,c_s$ are chosen as the Gaussian quadrature
nodes, i.e., the zeros of the $s$-th shifted Legendre polynomial $\frac{d^s}{dx^s}(x^s(x-1)^s)$ satisfies the condition \eqref{RK-symplectic-condition} \cite{ELW06}. In particular,
the coefficients of the Gauss methods of order 2, 4 and 6 can be given\cite{ELW06,Sanzs88}, respectively, (see Table. \ref{Gaussian23}).

\begin{table}[H]
\centering
\begin{tabular}{c|cc}
${c}$ & ${A}$  \\
\hline
& ${b}^{T}$ \\
\end{tabular}
=
\begin{tabular}{c|c}
$\frac{1}{2}$ &$\frac{1}{2}$ \\
\hline
                                 &1
\end{tabular},\ \begin{tabular}{c|cc}
${c}$ & ${A}$  \\
\hline
& ${b}^{T}$ \\
\end{tabular}
=
\begin{tabular}{c|cc}
$\frac{1}{2}-\frac{\sqrt{3}}{6}$ &$\frac{1}{4}$ & $\frac{1}{4}- \frac{\sqrt{3}}{6}$\\
$\frac{1}{2}+\frac{\sqrt{3}}{6}$ &$\frac{1}{4}+ \frac{\sqrt{3}}{6}$ &$\frac{1}{4}$ \\
\hline
                                 &$\frac{1}{2}$&$\frac{1}{2}$
\end{tabular},\\
\begin{tabular}{c|cc}
${c}$ & ${A}$  \\
\hline
& ${b}^{T}$ \\
\end{tabular}
=\begin{tabular}{c|ccc}
$\frac{1}{2}-\frac{\sqrt{15}}{10}$ &$\frac{5}{36}$ &  $\frac{2}{9}-\frac{\sqrt{15}}{15}$    &$\frac{5}{36}- \frac{\sqrt{15}}{30}$\\
$\frac{1}{2}$   &$\frac{5}{36}+ \frac{\sqrt{15}}{24}$  &  $\frac{2}{9}$  & $\frac{5}{36}- \frac{\sqrt{15}}{24}$ \\
$\frac{1}{2}+\frac{\sqrt{15}}{10}$ & $\frac{5}{36}+ \frac{\sqrt{15}}{30}$ & $\frac{2}{9}+\frac{\sqrt{15}}{15}$  & $\frac{5}{36}$ \\
\hline
& $\frac{5}{18}$ & $\frac{4}{9}$  & $\frac{5}{18}$
\end{tabular}.
\caption{\footnotesize The Gauss methods of 2 (s=1), 4 (s=2) and 6 (s=3).}\label{Gaussian23}
\end{table}

\end{rmk}
\begin{rmk} If the Gauss method of order 2 (see Table. \ref{Gaussian23}) is selected, the {\bf Scheme \ref{scheme:SPRK}} reduced to
the following semi-discrete Crank-Nicolson scheme (CNS)\cite{CaiYongyong22ANM}
\begin{equation}
\left\{
\begin{split}
& {\rm i} \delta_tE^n+\Delta E^{n+\frac{1}{2}}- \varepsilon^2 \Delta^2 E^{n+\frac{1}{2}}-N^{n+\frac{1}{2}} E^{n+\frac{1}{2}}=0, \\
&\delta_{t}N^n=\Delta v^{n+\frac{1}{2}},\\
&\delta_{t}v^n-N^{n+\frac{1}{2}}+\varepsilon^2 \Delta N^{n+\frac{1}{2}}-\frac{1}{2}(|E^{n+1}|^2+|E^{n}|^2)=0,\ n=0,1,2,\cdots,J,\\
 \end{split}\right.
\end{equation}
where
\begin{align*}
\delta_tw^n=\frac{w^{n+1}-w^n}{\tau},\ w^{n+\frac{1}{2}}=\frac{w^{n+1}+w^n}{2}.
\end{align*}
\end{rmk}

\begin{rmk} It is well-known that the scalar auxiliary variable (SAV) approach \cite{SXY18,SXY2019SR} is also an efficient method for developing high-order accurate structure-preserving methods for the conservative systems \cite{CWJ2021cpc,JCQSjsc2022}. However, we should note that it is challenging for introducing a special scalar auxiliary variable to construct high-order accurate methods in time which can preserve the original Hamiltonian energy of the system.
\end{rmk}

\subsection{Full discretization }

In this subsection, the Fourier pseudo-spectral method is employed to discretize the \textbf{Scheme \ref{scheme:SPRK}} in spatial variables and a class of fully discrete structure-preserving schemes are presented.

Let the domain $\Omega=[a,b)\times[c,d)$ be uniformly partition with spatial steps $h_x=\frac{b-a}{\mathcal{N}_x}$ and $h_y=\frac{d-c}{\mathcal{N}_y}$, where $\mathcal{N}_x$ and $\mathcal{N}_y$ are two positive even integers.
Then, we denote the spatial grid points as
\begin{equation*}
\begin{aligned}
\Omega_{h}=\{(x_j,y_k) |x_j=a+jh_x,\ y_k=c+kh_y,\ 0\leq   j\leq \mathcal{N}_x-1,\ 0\leq   k\leq \mathcal{N}_y-1\}.
\end{aligned}
\end{equation*}
Let $w_{j,k}$ be the numerical approximation of  $w(x_j,y_k,t)$ on $\Omega_{h}$, and denote
  $$w:=(w_{0,0},w_{1,0},\cdots,w_{\mathcal{N}_x-1,0},w_{0,1},w_{1,1},\cdots,w_{\mathcal{N}_x-1,1},\cdots,w_{0,\mathcal{N}_y-1},w_{1,\mathcal{N}_y-1},\cdots,w_{\mathcal{N}_x-1,\mathcal{N}_y-1})^{T}$$ be the solution vector; we also define discrete inner product and norms as
\begin{align*}
\langle u,  w \rangle_h=h_xh_y \sum\limits_{j=0}^{\mathcal{N_x}-1}\sum\limits_{k=0}^{\mathcal{N}_y-1} u_{ j,k}{{{\bar{w}}}}_{ j,k}, \ \ \
\| w\|_h=\langle w ,w\rangle_h^{\frac{1}{2}}, \ \
\| w\|_{h,{\infty}}=\max\limits_{(x_j,y_k)\in\Omega_{h}}|w_{j,k}|.
 \end{align*}
  In addition, we denote $`\cdot$' as the element-wise product of vectors ${u}$ and  ${w}$, that is
\begin{align*}
u\cdot w=&\big(u_{0,0}w_{0,0},\cdots,u_{\mathcal{N}_x-1,0}w_{\mathcal{N}_x-1,0},\cdots,u_{0,\mathcal{N}_y-1}w_{0,\mathcal{N}_y-1},
\cdots,u_{\mathcal{N}_x-1,\mathcal{N}_y-1}w_{\mathcal{N}_x-1,\mathcal{N}_y-1}\big)^{T}.
\end{align*}
For brevity, we denote ${w}\cdot w$ and ${w}\cdot\bar{w}$  as ${w}^2$ and $|w|^2$, respectively.

Let ${X}_{j}(x)$  and  ${Y}_{k}(y)$  be the interpolation basis functions given by
\begin{align*}
&{X}_{i}(x)=\frac{1}{\mathcal{N}_x}\sum\limits_{m={-\mathcal{N}_x}/{2}}^{{\mathcal{N}_x}/{2}}\frac{1}{a_m}e^{\text{i}m\mu_x(x-x_i)},
\ {Y}_{j}(y)=\frac{1}{\mathcal{N}_y}\sum\limits_{m={-\mathcal{N}_y}/{2}}^{{\mathcal{N}_y}/{2}}\frac{1}{b_m}e^{\text{i}m\mu_y(y-y_j)},
\end{align*}
where $\mu_x=\frac{2\pi}{b-a},\ \mu_y=\frac{2\pi}{d-c}$, $
a_m=\left\{\begin{array}{lll}
1,\ &\mid m \mid < \frac{\mathcal{N}_x}{2},\vspace{2mm}
\\
2,\ &\mid m \mid = \frac{\mathcal{N}_x}{2}.
\end{array}
\right.
$ and $
b_m=\left\{\begin{array}{lll}
1,\ &\mid m \mid < \frac{\mathcal{N}_y}{2},\vspace{2mm}
\\
2,\ &\mid m \mid = \frac{\mathcal{N}_y}{2}.
\end{array}
\right.
$
Then, we define $\mathcal{S}^{''}$ as the interpolation space
\begin{align*}
\mathcal{S}^{''}={\rm span}\big\{{X}_j(x){Y}_k(y) |~0\leq j\leq \mathcal{N}_x-1,\ 0\leq j\leq \mathcal{N}_y-1\big\},
\end{align*} and the interpolation operator $I_\mathbb{N}:C(\Omega)\rightarrow \mathcal{S}^{''}$ is given by \cite{chenjb01}
\begin{align}\label{GSAV:eq:3.1}
I_{\mathbb{N}} w(x,y,t)=\sum\limits_{j=0}^{\mathcal{N}_x-1}\sum\limits_{j=0}^{\mathcal{N}_y-1}w(x_j,y_k,t){X}_j(x){Y}_k(y).
\end{align}
Taking the second-order derivative with respect to $x$ and $y$, respectively and the
resulting expression at the collocation points $w(x_j,y_k,t)$ reads
\begin{align}\label{GSAV:eq:3.2}
&\frac{\partial^2}{\partial x^2}I_{\mathbb{N}}w(x_j,y_k,t)=\sum_{l=0}^{\mathcal N_x-1}w(x_j,y_k,t)\frac{d^2}{dx^2}{X}_l(x_j)=\sum_{l=0}^{\mathcal N_x-1}(\mathbf{D}_{2}^{x})_{j,l}w(x_j,y_k,t),\\
&\frac{\partial^2}{\partial y^2}I_{\mathbb{N}}w(x_j,y_k,t)=\sum_{l=0}^{\mathcal N_y-1}w(x_j,y_k,t)\frac{d^2}{dy^2}{Y}_l(y_k)=\sum_{l=0}^{\mathcal N_y-1}w(x_j,y_k,t)(\mathbf{D}_{2}^{y})_{k,l},
\end{align}
where \cite{chenjb01}
\begin{align*}
(\mathbf{D}_{2}^{x})_{j,k}=\left\{\begin{array}{lll}
\frac{1}{2}\mu^2_x(-1)^{j+k+1}\csc^2\big(\mu_x\frac{x_j-x_k}{2}\big),\ & j\neq k,
\\
-\mu^2_x\frac{\mathcal{N}_x^2+2}{12},\ & j=k,
\end{array}
\right.\\
(\mathbf{D}_{2}^{y})_{j,k}=\left\{\begin{array}{lll}
\frac{1}{2}\mu^2_y(-1)^{j+k+1}\csc^2\big(\mu_y\frac{y_j-y_k}{2}\big),\ & j\neq k,
\\ \\
-\mu^2_y\frac{\mathcal{N}_y^2+2}{12},\ & j=k.
\end{array}
\right.\\
\end{align*}
\begin{lem}\label{lem:diagonalization} ~For the matrix $\mathbf{D}_{2}^{\theta}$ ($\theta=x$ or $y$), there exists the following relation \cite{ST06}
\begin{align}\label{GSAV:eq:3.3}
\mathbf{D}_{2}^{\theta}=\mathcal{F}_{\mathcal{N}_{\theta}}^{H}\Lambda_{\theta}^{}\mathcal{F}_{\mathcal{N}_{\theta}},
\end{align}
where $\mathcal{F}_{\mathcal{N}_{\theta}}$ denotes the discrete Fourier transform (DFT) matrix, and satisfies $\mathcal{F}_{\mathcal{N}_{\theta}}^{H}$ is the conjugate transpose matrix of $\mathcal{F}_{\mathcal{N}_{\theta}}$,
\begin{align}\label{GSAV:eq:3.5}
\Lambda_{\theta}^{}=-\mu_{\theta}^2\text{diag}\big[0^2, 1^2, \cdots, (\frac{\mathcal{N_{\theta}}}{2})^2,(-\frac{\mathcal{N_{\theta}}}{2}+1)^2,\cdots, (-2)^2,(-1)^2\big].
\end{align}
\end{lem}

Then, applying the Fourier pseudo-spectral method to the \textbf{Scheme \ref{scheme:SPRK}}, we obtain the following fully discrete scheme.

\begin{scm}\label{scheme:RD-QAV-EPRK} Let $b_i,a_{ij}({i,j=1,\cdots,s})$ be real numbers and let $c_i=\sum_{j=1}^sa_{ij}$.
 For given $(E^{n},N^{n}, v^{n})$, an $s$-stage RK Fourier pseudo-spectral method is given by
		\begin{equation} 
			\begin{cases}
				E_{ni} =  E^n +  \tau \sum\limits_{j=1}^s a_{ij} k_j^{1}, \\
            k_i^{1} = {\rm i} ( \Delta_h E_{ni}- \varepsilon^2 \Delta_h^2 E_{ni} -N_{ni}\cdot E_{ni} ), \\
	    	N_{ni} =  N^n +  \tau \sum\limits_{j=1}^s a_{ij} k_j^{2}, \\
            k_i^{2} =  \Delta_h v_{ni}, \\
			v_{ni} = v^n +  \tau \sum\limits_{j=1}^s a_{ij}  k_j^{3}, \\
            k_i^{3}=   N_{ni} - \varepsilon^2 \Delta_h  N_{ni} + |E^n|^2 + 2 \tau \sum\limits_{j=1}^sa_{ij} {\rm Re} ( \bar{E}_{nj}\cdot k_j^{1}),
			\end{cases}
		\end{equation}
where $\Delta_h= \mathbf{I}_y\otimes\mathbf{D}_{x}^{2}+\mathbf{D}_{y}^{2}\otimes\mathbf{I}_x $.
		Then $(E^{n+1},N^{n+1}, v^{n+1})$ is updated by
		\begin{align}
		&	E^{n+1} =  E^n +  \tau \sum\limits_{i=1}^s b_{i} k_i^{1}, \label{QAVRK-E} \\
		&N^{n+1} =  N^n +  \tau \sum\limits_{i=1}^s b_{i} k_i^{2},  \label{QAVRK-N} \\
		&v^{n+1} =  v^n +  \tau \sum\limits_{i=1}^s b_{i} k_i^{3}, \label{QAVRK-v}
		\end{align}
with the initial conditions
\begin{align}\label{intial-data}
		 E^{0} =  E_0({\bf x}), \
		 N^{0} =  N_0({\bf x}), \
		  \Delta_h  v^{0} =  N_1({\bf x}).
		\end{align}
We note that based on the condition  $\int_{\Omega} v(\mathbf{x},0) d \mathbf{x} =0$, we can obtain the zeroth Fourier coefficient on $v^0$ is zero, so that the Poisson equation $ \Delta_h  v^{0} =  N_1({\bf x})$ with periodic boundary conditions only has a unique  solution.
	\end{scm}

For the {\bf Scheme \ref{scheme:RD-QAV-EPRK}}, we can obtain the following theorem which can be carried out similarly as \textbf{Theorem \ref{thm:SD-CL}} and \textbf{Theorem \ref{thm:original-energy-conservation}}, respectively.

\begin{thm}\label{thm:FD-CL}
	If the RK coefficients of the {\bf Scheme \ref{scheme:RD-QAV-EPRK}} satisfy the condition \eqref{RK-symplectic-condition}, then {\bf Scheme \ref{scheme:RD-QAV-EPRK}} satisfies the following discrete conservative properties
\begin{itemize}
\item The discrete mass
\begin{align}\label{QAVRKMCL-FD}
\mathcal{M}_h^{n+1}=\mathcal{M}_h^{n},\ \mathcal{M}_h^{n}=\big\langle E^{n},E^{n}\big\rangle_h,\ n=0,1,2,\cdots,J.
\end{align}
\item The discrete quadratic invariant
\begin{align}\label{QAVRKCL-FD}
&q^n-|E^{n}|^2={\bf 0},\ n=0,1,2,\cdots,J.
\end{align}
\item The discrete Hamiltonian energy
\begin{align}\label{QAVRKECL-FD}
& \mathcal{H}_h^{n+1} = \mathcal{H}_h^n,\ n=0,1,2,\cdots,J,
	\end{align}
where
\begin{align*}
\mathcal{H}_h^n = &\langle \Delta_h E^n, E^n \rangle_h-\varepsilon^{2} \langle \Delta_h E^n,\Delta_h E^n\rangle_h\\
& -\frac{1}{2} \langle N^n,N^n\rangle_h
     + \frac{\varepsilon^{2}}{2} \langle \Delta_h N^n,N^n\rangle_h -\langle N^n, |E^n|^n\rangle_h + \frac{1}{2} \langle \Delta_h v^n, v^n\rangle_h.
\end{align*}
\end{itemize}
     \end{thm}

\section{An efficient implementation for the proposed scheme}

In this section, we will present an efficient fixed point iteration method to solve the resulting nonlinear equations of the {\bf Scheme \ref{scheme:RD-QAV-EPRK}}, which is mainly based on the matrix diagonalization
method (see Ref. \cite{ST06} and references therein) and the discrete Fourier transform
 algorithm. For simplicity, we only consider the one dimensional QZS \eqref{eq1.1} and generalizations to two dimensional case is straightforward for tensor product grids and the results remain valid with modifications. In addition, we only take the 2-stage Gauss method (i.e., $s=2$) as an example in which the RK coefficients $a_{ij},\ b_i \ (i,j=1,2,\cdots,s)$ are given in Table \ref{Gaussian23}. 

For given $E^n,\ N^n$ and $ v^n$, 2-stage Gauss method is equivalent to
\begin{align}
&k_1^1=\text{i}(\Delta_h-\varepsilon^2\Delta_h^2)E_{n1}-\text{i}N_{n1}\cdot E_{n1},\quad k_2^1=\text{i}(\Delta_h-\varepsilon^2\Delta_h^2)E_{n2}-\text{i}N_{n2}\cdot E_{n2},\label{f-s-1} \\
& k_1^2=\Delta_hv_{n1},\hspace{4cm}   k_2^2=\Delta_hv_{n2},\label{f-s-2}  \\
&k_1^3=(1-\varepsilon^2\Delta_h)N_{n1}+Q_1,\hspace{1.85cm}     k_2^3=(1-\varepsilon^2\Delta_h)N_{n2}+Q_2,\label{f-s-3} \\
&E_{n1} = E^n+\tau a_{11}k_1^1+\tau a_{12}k_2^1,\hspace{1.2cm}  E_{n2} = E^n+\tau a_{21}k_1^1+\tau a_{22}k_2^1,\label{f-s-5}\\
&N_{n1} = N^n+\tau a_{11}k_1^2+\tau a_{12}k_2^2,\hspace{1.15cm}  N_{n2} = N^n+\tau a_{21}k_1^2+\tau a_{22}k_2^2,\label{f-s-6}\\
&v_{n1} = v^n+\tau a_{11}k_1^3+\tau a_{12}k_2^3,\hspace{1.5cm} v_{n2} = v^n+\tau a_{21}k_1^3+\tau a_{22}k_2^3,\label{f-s-7}
\end{align}
where
\begin{align}
&Q_1 = |E^n|^2+\tau a_{11}k_1^4+\tau a_{12}k_2^4,\  Q_2 = |E^n|^2+\tau a_{21}k_1^4+\tau a_{22}k_2^4,\label{f-s-8} \\
&k_1^4=2\text{Re}(\bar{E}_{n1}\cdot k_1^1),~~~~~~\hspace{1.2cm}  k_2^4=2\text{Re}(\bar{E}_{n2}\cdot k_2^1). \label{f-s-4}
\end{align}
Then, $E^{n+1},\ N^{n+1}$ and $v^{n+1}$ are obtained by
\begin{align}\label{n-f-s-9}
&E^{n+1}=E^n+\tau b_1k_1^1+\tau b_2k_2^1,\ N^{n+1}=N^n+\tau b_1k_1^2+\tau b_2k_2^2,\\\label{n-f-s-10}
&v^{n+1}=v^{n}+\tau b_1k_1^3+\tau b_2k_2^3.
\end{align}
With \eqref{f-s-1} and \eqref{f-s-5}, we have
\begin{align}
&\left[
  \begin{array}{cc}\vspace{2mm}
    1-\tau\text{i}(\Delta_h-\varepsilon^2\Delta_h^2)a_{11} & -\tau\text{i}(\Delta_h-\varepsilon^2\Delta_h^2)a_{12}\\
    -\tau\text{i}(\Delta_h-\varepsilon^2\Delta_h^2)a_{21} & 1-\tau\text{i}(\Delta_h-\varepsilon^2\Delta_h^2)a_{22}\\
  \end{array}
\right]\left[
  \begin{array}{c} \vspace{2mm}
  k_1^1\\
  k_2^1
\end{array}
  \right]\nonumber\\\label{f-s-9}
  &~~~~~~~~~~~~~~~~~~~~~~~~~~~~~~~~~~~~~~~~~~~~~~~~~~~~~~~~~~~=\left[
  \begin{array}{c}\vspace{2mm}
  \text{i}(\Delta_h-\varepsilon^2\Delta_h^2)E^n-\text{i}N_{n1}\cdot E_{n1}\\
  \text{i}(\Delta_h-\varepsilon^2\Delta_h^2)E^n-\text{i}N_{n2}\cdot E_{n2}
\end{array}
  \right].
\end{align}
Analogously, we can obtain from \eqref{f-s-2}-\eqref{f-s-3} and \eqref{f-s-6}-\eqref{f-s-7} that
\begin{align}
&\left[
  \begin{array}{cc}\vspace{2mm}
    1-\tau^2(a_{11}^2+a_{12}a_{21})(\Delta_h-\varepsilon^2\Delta_h^2) & -\tau^2(a_{11}a_{12}+a_{12}a_{22})(\Delta_h-\varepsilon^2\Delta_h^2)\\
    -\tau^2(a_{21}a_{11}+a_{22}a_{21})(\Delta_h-\varepsilon^2\Delta_h^2) & 1-\tau^2(a_{21}a_{12}+a_{22}^2)(\Delta_h-\varepsilon^2\Delta_h^2)\\
  \end{array}
\right]\left[
  \begin{array}{c} \vspace{2mm}
  k_1^2\\
  k_2^2
\end{array}
  \right]\nonumber\\\label{f-s-10}
  &~~~~~~~~~~~~~~~~~~~~~~~~~~~~~~=\left[
  \begin{array}{c}\vspace{2mm}
  \tau(a_{11}+a_{12})(\Delta_h-\varepsilon^2\Delta_h^2)N^n+\Delta_h v^n+\tau a_{11}\Delta_h Q_1+\tau a_{12}\Delta_h Q_2\\
  \tau(a_{21}+a_{22})(\Delta_h-\varepsilon^2\Delta_h^2)N^n+\Delta_h v^n+\tau a_{21}\Delta_h Q_1+\tau a_{22}\Delta_h Q_2
\end{array}
  \right].
\end{align}
It is clear to see that if $k_i^1$ and $k_i^2 \ (i=1,2)$ are obtained after solving the nonlinear equations \eqref{f-s-9} and \eqref{f-s-10}, respectively, and $E_{ni}\ (i=1,2)$ and $N_{ni}\ (i=1,2)$ are updated by \eqref{f-s-5} and \eqref{f-s-6}, respectively. Subsequently, $k_i^3 \ (i=1,2)$ and $k_i^4 \ (i=1,2)$ are updated by \eqref{f-s-3} and \eqref{f-s-4}, respectively. Thus, for the 2-stage Gauss method, we only need to solve the nonlinear equations \eqref{f-s-9} and \eqref{f-s-10}, respectively. For the nonlinear algebraic equations \eqref{f-s-9} and \eqref{f-s-10}, we apply the following fixed-point iteration strategy,
\begin{align*}
&\left[
  \begin{array}{cc}\vspace{2mm}
    1-\tau\text{i}(\Delta_h-\varepsilon^2\Delta_h^2)a_{11} & -\tau\text{i}(\Delta_h-\varepsilon^2\Delta_h^2)a_{12}\\
    -\tau\text{i}(\Delta_h-\varepsilon^2\Delta_h^2)a_{21} & 1-\tau\text{i}(\Delta_h-\varepsilon^2\Delta_h^2)a_{22}\\
  \end{array}
\right]\left[
  \begin{array}{c} \vspace{2mm}
  (k_1^1)^{l+1}\\
  (k_2^1)^{l+1}
\end{array}
  \right]\nonumber\\
  &~~~~~~~~~~~~~~~~~~~~~~~~~~~~~~~~~~~~~~~~~~~~~~~~~~~~~~~~~~~~~~~~~~~~~~~~~~~~~~~~=\left[
  \begin{array}{c}\vspace{2mm}
  \text{i}(\Delta_h-\varepsilon^2\Delta_h^2)E^n-\text{i}N_{n1}^{l}\cdot E_{n1}^{l}\\
  \text{i}(\Delta_h-\varepsilon^2\Delta_h^2)E^n-\text{i}N_{n2}^{l}\cdot E_{n2}^{l}
\end{array}
  \right],
\end{align*}
and
\begin{align*}
&\left[
  \begin{array}{cc}\vspace{2mm}
    1-\tau^2(a_{11}^2+a_{12}a_{21})(\Delta_h-\varepsilon^2\Delta_h^2) & -\tau^2(a_{11}a_{12}+a_{12}a_{22})(\Delta_h-\varepsilon^2\Delta_h^2)\\
    -\tau^2(a_{21}a_{11}+a_{22}a_{21})(\Delta_h-\varepsilon^2\Delta_h^2) & 1-\tau^2(a_{21}a_{12}+a_{22}^2)(\Delta_h-\varepsilon^2\Delta_h^2)\\
  \end{array}
\right]\left[
  \begin{array}{c} \vspace{2mm}
  (k_1^2)^{l+1}\\
  (k_2^2)^{l+1}
\end{array}
  \right]\nonumber\\
  &~~~~~~~~~~~~~~~~~~~~~~~~~~~~~~~~~~~~~~~=\left[
  \begin{array}{c}\vspace{2mm}
  \tau(a_{11}+a_{12})(\Delta_h-\varepsilon^2\Delta_h^2)N^n+\Delta_h v^n+\tau a_{11}\Delta_h Q_1^{l}+\tau a_{12}\Delta_h Q_2^{l}\\
  \tau(a_{21}+a_{22})(\Delta_h-\varepsilon^2\Delta_h^2)N^n+\Delta_h v^n+\tau a_{21}\Delta_h Q_1^{l}+\tau a_{22}\Delta_h Q_2^{l}
\end{array}
  \right],
\end{align*}
where
\begin{align*}
&E_{n1}^{l} = E^n+\tau a_{11}(k_1^1)^{l}+\tau a_{12}(k_2^1)^{l},\quad E_{n2}^{l} = E^n+\tau a_{21}(k_1^1)^{l}+\tau a_{22}(k_2^1)^{l},\\
&N_{n1}^{l} = N^n+\tau a_{11}(k_1^2)^{l}+\tau a_{12}(k_2^2)^{l},\quad  N_{n2}^{l} = N^n+\tau a_{21}(k_1^2)^{l}+\tau a_{22}(k_2^2)^{l},\\
&Q_{1}^{l} = |E^n|^2+2\tau a_{11}\text{Re}(\bar{E}_{n1}^l\cdot (k_1^1)^l)+2\tau a_{12}\text{Re}(\bar{E}_{n2}^l\cdot (k_2^1)^l),\\
&Q_{2}^{l} = |E^n|^2+2\tau a_{21}\text{Re}(\bar{E}_{n1}^l\cdot (k_1^1)^l)+2\tau a_{22}\text{Re}(\bar{E}_{n2}^l\cdot (k_2^1)^l).
\end{align*}
Then, it follows from \textbf{Lemma \ref{lem:diagonalization}} that
\begin{align}\label{f-s-4.13}
&\left[
  \begin{array}{cc}\vspace{2mm}
    1-\tau\text{i}(\Lambda_x-\varepsilon^2\Lambda_x^2)a_{11} & -\tau\text{i}(\Lambda_x-\varepsilon^2\Lambda_x^2)a_{12}\\
    -\tau\text{i}(\Lambda_x-\varepsilon^2\Lambda_x^2)a_{21} & 1-\tau\text{i}(\Lambda_x-\varepsilon^2\Lambda_x^2)a_{22}\\
  \end{array}
\right]\left[
  \begin{array}{c} \vspace{2mm}
  \widetilde{(k_1^1)^{l+1}}\\
  \widetilde{(k_2^1)^{l+1}}
\end{array}
  \right]\nonumber\\
  &~~~~~~~~~~~~~~~~~~~~~~~~~~~~~~~~~~~~~~~~~~~~~~~~~~~~~~~~~~~~~~~~~=\left[
  \begin{array}{c}\vspace{2mm}
  \text{i}(\Lambda_x-\varepsilon^2\Lambda_x^2)\widetilde{ E^n}-\widetilde{\text{i}N_{n1}^{l}\cdot E_{n1}^{l}}\\
  \text{i}(\Lambda_x-\varepsilon^2\Lambda_x^2)\widetilde{ E^n}-\widetilde{\text{i}N_{n2}^{l}\cdot E_{n2}^{l}}
\end{array}
  \right],
\end{align}
and
\begin{align}\label{f-s-4.14}
&\left[
  \begin{array}{cc}\vspace{2mm}
    1-\tau^2(a_{11}^2+a_{12}a_{21})(\Lambda_x-\varepsilon^2\Lambda_x^2) & -\tau^2(a_{11}a_{12}+a_{12}a_{22})(\Lambda_x-\varepsilon^2\Lambda_x^2)\\
    -\tau^2(a_{21}a_{11}+a_{22}a_{21})(\Lambda_x-\varepsilon^2\Lambda_x^2) & 1-\tau^2(a_{21}a_{12}+a_{22}^2)(\Lambda_x-\varepsilon^2\Lambda_x^2)\\
  \end{array}
\right]\left[
  \begin{array}{c} \vspace{2mm}
  \widetilde{(k_1^2)^{l+1}}\\
  \widetilde{(k_2^2)^{l+1}}
\end{array}
  \right]\nonumber\\
  &~~~~~~~~~~~~~~~~~~~~~~~~~~~~~=\left[
  \begin{array}{c}\vspace{2mm}
  \tau(a_{11}+a_{12})(\Lambda_x-\varepsilon^2\Lambda_x^2)\widetilde{ N^n}+\Lambda_x \widetilde{ v^n}+\tau a_{11}\Lambda_x \widetilde{ Q_1^{l}}+\tau a_{12}\Lambda_x \widetilde{ Q_2^{l}}\\
  \tau(a_{21}+a_{22})(\Lambda_x-\varepsilon^2\Lambda_x^2)\widetilde{ N^n}+\Lambda_x \widetilde{ v^n}+\tau a_{21}\Lambda_x \widetilde{ Q_1^{l}}+\tau a_{22}\Lambda_x \widetilde{ Q_2^{l}}
\end{array}
  \right],
\end{align}
where $\widetilde{\bullet}=\mathcal{F}_{\mathcal{N}_{x}}\bullet$ and $\mathcal{F}_{\mathcal{N}_{x}}$ is the discrete Fourier transform  matrix. Then we rewrite \eqref{f-s-4.13} and \eqref{f-s-4.14} into component-wise form, as follows:
\begin{align*}
&\left[
  \begin{array}{cc}\vspace{2mm}
    1-\tau\text{i}(\Delta_{j}-\varepsilon^2\Delta_{j}^2)a_{11} & -\tau\text{i}(\Delta_{j}-\varepsilon^2\Delta_{j}^2)a_{12}\\
    -\tau\text{i}(\Delta_{j}-\varepsilon^2\Delta_{j}^2)a_{21} & 1-\tau\text{i}(\Delta_{j}-\varepsilon^2\Delta_{j}^2)a_{22}\\
  \end{array}
\right]\left[
  \begin{array}{c} \vspace{2mm}
  \widetilde{ (k_1^1)_j^{l+1}}\\
  \widetilde{ (k_2^1)_j^{l+1}}
\end{array}
  \right]\nonumber\\
  &~~~~~~~~~~~~~~~~~~~~~~~~~~~~~~~~~~~~~~~~~~~~~~~~~~~~~~~~~~~~~~=\left[
  \begin{array}{c}\vspace{2mm}
  \text{i}(\Delta_{j}-\varepsilon^2\Delta_{j}^2)\widetilde{ E^n_j}-(\widetilde{\text{i}N_{n1}^{l}\cdot E_{n1}^{l}})_j\\
  \text{i}(\Delta_{j}-\varepsilon^2\Delta_{j}^2)\widetilde{ E^n_j}-(\widetilde{\text{i}N_{n2}^{l}\cdot E_{n2}^{l}})_j
\end{array}
  \right],
\end{align*}
and
\begin{align*}
&\left[
  \begin{array}{cc}\vspace{2mm}
    1-\tau^2(a_{11}^2+a_{12}a_{21})(\Delta_{j}-\varepsilon^2\Delta_{j}^2) & -\tau^2(a_{11}a_{12}+a_{12}a_{22})(\Delta_{j}-\varepsilon^2\Delta_{j}^2)\\
    -\tau^2(a_{21}a_{11}+a_{22}a_{21})(\Delta_{j}-\varepsilon^2\Delta_{j}^2) & 1-\tau^2(a_{21}a_{12}+a_{22}^2)(\Delta_{j}-\varepsilon^2\Delta_{j}^2)\\
  \end{array}
\right]\left[
  \begin{array}{c} \vspace{2mm}
  \widetilde{ (k_1^2)_j^{l+1}}\\
  \widetilde{ (k_2^2)_j^{l+1}}
\end{array}
  \right]\nonumber\\
  &~~~~~~~~~~~~~~~~~~~~~=\left[
  \begin{array}{c}\vspace{2mm}
  \tau(a_{11}+a_{12})(\Delta_{j}-\varepsilon^2\Delta_{j}^2)\widetilde{ N^n_j}+\Delta_{j}  \widetilde{ v^n_j}+\tau a_{11}\Delta_{j} \widetilde{ (Q_1^{l})_j}+\tau a_{12}\Delta_{j} \widetilde{ (Q_2^{l})_j}\\
  \tau(a_{21}+a_{22})(\Delta_{j}-\varepsilon^2\Delta_{j}^2)\widetilde{ N^n_j}+\Delta_{j}  \widetilde{ v^n_j}+\tau a_{21}\Delta_{j}  \widetilde{ (Q_1^{l})_j}+\tau a_{22}\Delta_{j} \widetilde{ (Q_2^{l})_j}
\end{array}
  \right],
\end{align*}
where $\Delta_{j}=(\Lambda_x)_j$ and $j=0,1,2,\cdots,\mathcal{N}_x-1$.

Solving above two linear systems for all $j=0,1,2,\cdots,\mathcal{N}_x-1$, we obtain $\widetilde{ (k_i^1)^{l+1}}$ and $\widetilde{ (k_i^2)^{l+1}}$, then $(k_i^1)^{l+1}$ and $(k_i^2)^{l+1}$ are updated by
$\mathcal{F}_{\mathcal{N}_{x}}^H\widetilde{ (k_i^1)^{l+1}}$ and $\mathcal{F}_{\mathcal{N}_{x}}^H\widetilde{ (k_i^2)^{l+1}}\ (i=1,2),$ respectively. In our computations, we take the iterative initial value $(k_i^1)^{0}=E^n$ and $(k_i^2)^{0}=N^n,\ i=1,2$.
The iteration terminates when  the number of maximum iterative step  $M = 30$  is reached or the infinity norm of the error
between two adjacent iterative steps is less than $10^{-14}$, that is
\begin{align*}
\mathop{\rm max}\Big \{ \mathop{\rm max}\limits_{l\leqslant i\leqslant 2} \big \{  \lVert ({k}_i^1)^{l+1}-({k}_i^1)^l \rVert_{h,\infty} \big \},\mathop{\rm max}\limits_{l\leqslant i\leqslant 2} \big \{  \lVert ({k}_i^2)^{l+1}-({k}_i^2)^l \rVert_{h,\infty} \big \}\Big \}  < 10^{-14}.
\end{align*}
   After $k_i^1$ and $k_i^2,\ i=1,2$ are obtained, we then have
 \begin{itemize}
 \item Step 1: $E^{n+1}$ and $N^{n+1}$ are obtained from \eqref{n-f-s-9},  $E_{ni}$ and $N_{ni}  \ (i=1,2)$ are obtained from \eqref{f-s-5} and \eqref{f-s-6};
 \item Step 2: We first obtain $k_i^4$ from \eqref{f-s-4} and $Q_i \ (i=1,2)$ is then calculated by \eqref{f-s-8};
 \item Step 3: $k_i^3 \ (i=1,2)$ are then obtained from \eqref{f-s-3} and $v^{n+1}$ is  updated by \eqref{n-f-s-10}.
\end{itemize}

{\begin{rmk} Let $\widetilde{\bullet}=\mathcal{F}_{\mathcal{N}_{x}}\bullet\mathcal{F}_{\mathcal{N}_{y}}^{H}$ where $\bullet$ represents the matrix of $\mathcal{N}_{x}\times \mathcal{N}_{y}$. By the similar argument as the 1D case, the efficient implementation of the proposed scheme for the 2D case can be expressed into the following component-wise form, as $s=2$ is chosen:
\begin{align*}
&\left[
  \begin{array}{cc}\vspace{2mm}
    1-\tau\text{i}\big(\Delta_{j,k}-\varepsilon^2\Delta_{j,k}^2\big)a_{11} & -\tau\text{i}(\Delta_{j,k}-\varepsilon^2\Delta_{j,k}^2)a_{12}\\
    -\tau\text{i}(\Delta_{j,k}-\varepsilon^2\Delta_{j,k}^2)a_{21} & 1-\tau\text{i}(\Delta_{j,k}-\varepsilon^2\Delta_{j,k}^2)a_{22}\\
  \end{array}
\right]\left[
  \begin{array}{c} \vspace{2mm}
  \Big(\widetilde{ (k_1^1)^{l+1}}\Big)_{j,k}\\
  \Big(\widetilde{ (k_2^1)^{l+1}}\Big)_{j,k}
\end{array}
  \right]\nonumber\\
  &~~~~~~~~~~~~~~~~~~~~~~~~~~~~~~~~~~~~~~~~~~~~~~~~~~~~~~~~~~=\left[
  \begin{array}{c}\vspace{2mm}
  \text{i}(\Delta_{j,k}-\varepsilon^2\Delta_{j,k}^2)(\widetilde{E^n})_{j,k}-\Big(\widetilde{\text{i}N_{n1}^{l}\cdot E_{n1}^{l}}\Big)_{j,k}\\
  \text{i}(\Delta_{j,k}-\varepsilon^2\Delta_{j,k}^2)(\widetilde{ E^n})_{j,k}-\Big(\widetilde{\text{i}N_{n2}^{l}\cdot E_{n2}^{l}}\Big)_{j,k}
\end{array}
  \right],
\end{align*}
and
\begin{align*}
&\left[
  \begin{array}{cc}\vspace{2mm}
    1-\tau^2(a_{11}^2+a_{12}a_{21})(\Delta_{j,k}-\varepsilon^2\Delta_{j,k}^2) & -\tau^2(a_{11}a_{12}+a_{12}a_{22})(\Delta_{j,k}-\varepsilon^2\Delta_{j,k}^2)\\
    -\tau^2(a_{21}a_{11}+a_{22}a_{21})(\Delta_{j,k}-\varepsilon^2\Delta_{j,k}^2) & 1-\tau^2(a_{21}a_{12}+a_{22}^2)(\Delta_{j,k}-\varepsilon^2\Delta_{j,k}^2)\\
  \end{array}
\right]\left[
  \begin{array}{c} \vspace{2mm}
  \Big(\widetilde{ (k_1^2)^{l+1}}\Big)_{j,k}\\
  \Big(\widetilde{ (k_2^2)_j^{l+1}}\Big)_{j,k}
\end{array}
  \right]\nonumber\\
  &~~~~~=\left[
  \begin{array}{c}\vspace{2mm}
  \tau(a_{11}+a_{12})\Big(\Delta_{j,k}-\varepsilon^2\Delta_{j,k}^2\Big)(\widetilde{ N^n})_{j,k}+\Delta_{j,k}(\widetilde{v^n})_{j,k}+\tau a_{11}\Delta_{j,k} (\widetilde{ (Q_1^{l})})_{j,k}+\tau a_{12}\Delta_{j,k}(\widetilde{ Q_2^{l}})_{j,k}\\
  \tau(a_{21}+a_{22})(\Delta_{j,k}-\varepsilon^2\Delta_{j,k}^2)(\widetilde{N^n})_{j,k}+\Delta_{j,k}(\widetilde{ v^n})_{j,k}+\tau a_{21}\Delta_{j,k}(\widetilde{Q_1^{l}})_{j,k}+\tau a_{22}\Delta_{j,k}(\widetilde{Q_2^{l}})_{j,k}
\end{array}
  \right],
\end{align*}
where $\Delta_{j,k}=(\Lambda_x)_j+(\Lambda_y)_k$, $j=0,1,2,\cdots,\mathcal{N}_x-1,\ k=0,1,2,\cdots,\mathcal{N}_y-1$. After we obtain $\widetilde{ (k_i^1)^{l+1}}$ and $\widetilde{ (k_i^2)^{l+1}}$, then $(k_i^1)^{l+1}$ and $(k_i^2)^{l+1}$ are updated by
$\mathcal{F}_{\mathcal{N}_{x}}^H\widetilde{ (k_i^1)^{l+1}}\mathcal{F}_{\mathcal{N}_{y}}$ and $\mathcal{F}_{\mathcal{N}_{x}}^H\widetilde{ (k_i^2)^{l+1}}\mathcal{F}_{\mathcal{N}_{y}}\ (i=1,2),$ respectively. We should note that the Fast Fourier Transform (FFT) algorithm is employed to the above process.

\end{rmk}}

\section{Numerical results}

The purpose of this section is to test validity and efficiency of the newly proposed schemes  for solving the QZS \eqref{eq1.1}
with the periodic boundary condition. In particular, when the $s$-stage Gauss method (see Remark \ref{rmk-ZQs-3.1}) is employed for the {\bf Scheme \ref{scheme:RD-QAV-EPRK}}, we denote the resulting structure-preserving RK (SPRK) method as SPRK-$s$. For brevity, in the rest of this paper, 2 and 3-stage Gauss methods (see Table. \ref{Gaussian23}) are only used for demonstration purposes. Also, the results are compared with the existing TS-EWI-FP method \cite{ZhangG21} and CNS \cite{CaiYongyong22ANM}, where the Fourier pseudo-spectral method takes the place of the compact finite difference method in space. Numerical examples including accuracy tests, convergence rates of the QZS \eqref{eq1.1} to the
classical ZS in the semi-classical limit, soliton collisions and pattern dynamics of the QZS \eqref{eq1.1} in 1D. Let $E_{h,\tau}^n$ and $N_{h,\tau}^n$ be the numerical solutions of $E(\cdot,t_n)$ and $N(\cdot,t_n)$ obtained by SPRK-2 and SPRK-3 with the spatial step $h$ and the time step $\tau$, respectively.
To quantify the numerical methods, we define the error
functions as, respectively
\[
e_{\varepsilon}(t_n) = \| E(\cdot,t_n) - E_{h,\tau}^n \|_{h,\infty}, \quad
n_{\varepsilon}(t_n) = \| N(\cdot,t_n) - N_{h,\tau}^n \|_{h,\infty}.
\]
Meanwhile, we also define the relative residuals on the  mass and energy as, respectively
\[RM^{n} = \Big| \frac{\mathcal{M}_h^n - \mathcal{M}_h^0}{\mathcal{M}_h^0}  \Big|,\
RH^{n} = \Big| \frac{\mathcal{H}_h^n - \mathcal{H}_h^0 }{\mathcal{H}_h^0} \Big|.
\]

\subsection{Accuracy test}
{ \bf Example 5.1.} As we choose $\varepsilon=0$, the one-dimensional QZS \eqref{eq1.1} reduces to the classical ZS (cf. \cite{Zakharov72}), which admits a solitary-wave solution given by \cite{Glassey92}
\begin{equation}
\label{eq5.2}
\begin{split}
&E(x, t)= {\rm i} \sqrt{2 B^{2}\left(1-V^{2}\right)}\,\sech\left(B\left(x-x_{0}-V t\right)\right)
   e^{{\rm i} \left(V\left(x-x_{0}\right) / 2-\left(V^{2} / 4-B^{2}\right) t\right)},\\
&N(x, t)=-2 B^{2} \,\sech^{2}\left(B\left(x-x_{0}-V t\right)\right),
\end{split}
\end{equation}
where $B$ is a constant, $x_{0}$ and $V$ represent  the initial displacement and
 the propagation velocity of the soliton, respectively.
In this test, the computations are done on the interval $\Omega= [-128,128)$, and we choose the parameters $B = 1$, $V = \frac{1}{ 2} $,  $x_{0} = 0$,
as well as the following initial conditions
\begin{equation}
\label{eq5.3}
E_{0}(x)={\rm i} \sqrt{1.5} \,\sech(x) e^{{\rm i}x / 4}, \quad N_{0}(x)=-2 \,\sech^{2}(x), \quad
N_{1}(x)=-2 \,\sech^2 (x) \tanh(x).
\end{equation}

\begin{table}[H]
\caption{Spatial errors provided by SPRK-2 and SPRK-3 at $T=1$ for the  classical ZS with the initial conditions \eqref{eq5.3}.}\label{tabs1}
\setlength{\tabcolsep}{3mm}
\centering
\begin{tabular}{cccccccc} \toprule
 Scheme  &  $\varepsilon=0$ & $h_0 = 1   $ & $h_0/2 $  & $h_0/2^2 $ & $h_0/2^3 $   \\
\midrule
SPRK-$2$         &  $e_{0}(1)$  & 4.58e-2 &  8.90e-5  & 1.38e-9   &3.25e-15\\
  &  $n_{0}(1)$  &  9.37e-2 &  7.15e-4  & 1.77e-8  & 5.67e-13 \\
\midrule \midrule
SPRK-$3$  & $e_{0}(1)$  & 4.58e-2 &   8.90e-5 &   1.38e-9 &   3.91e-15 \\
   &  $n_{0}(1)$  & 9.37e-2  &  7.15e-4  &  1.77e-8  &  7.57e-13 \\
\bottomrule
  \end{tabular}
  \end{table}

\begin{table}[H]
\caption{Temporal errors provided by SPRK-2 and SPRK-3 at $T=1$ for the  classical ZS with the initial conditions \eqref{eq5.3}.}\label{tabt1}
\setlength{\tabcolsep}{3mm}
\centering
\begin{tabular}{cccccccc} \toprule
 Scheme &  $\varepsilon=0$ & $\tau_0 = 1/4  $ & $\tau_0/2 $  & $\tau_0/2^2 $ & $\tau_0/2^3 $  & $\tau_0/2^4 $    \\
\midrule
                &  $e_{0}(1)$  & 2.57e-05  &  1.34e-06   &  8.03e-08   & 4.98e-09   & 3.11e-10  \\
SPRK-$2$     &  rate    &   -     & 4.27  &  4.06  &  4.01   & 4.00     \\
         &  $n_{0}(1)$   &   4.04e-05   & 2.50e-06   & 1.56e-07  &   9.77e-09  &  6.10e-10   \\
                  &  rate     &   -   & 4.02  & 4.00  &  4.01  &  4.00     \\
\midrule \midrule
               & $e_{0}(1)$   & 8.54e-07 &  1.07e-08  & 1.13e-10  &  1.58e-12  & 2.41e-14   \\
SPRK-$3$   &  rate     &   -   & 6.32 &   6.55 &   6.17  & 6.03    \\
       &  $n_{0}(1)$  &2.10e-07  &  3.43e-09 &  5.47e-11  & 8.62e-13 &  1.49e-14  \\
               &  rate    &   -  &   5.94 & 5.97 &  5.99 &  5.85     \\
\bottomrule
  \end{tabular}
  \end{table}

 Table  \ref{tabs1} reports the spatial errors
of SPRK-2 and SPRK-3 at $T = 1$  with a very small
time step $\tau=1/2^{12}$ such that the discretization error in time is negligible.
We observe that the spatial errors converge exponentially. Then to test the temporal discretization errors of the numerical schemes, we fix the Fourier node 2048
so that spatial errors play no role here. In Table \ref{tabt1}, we lists the temporal errors of SPRK-2 and SPRK-3 for the classical ZS at $T = 1$ with different time steps, respectively, which shows that
SPRK-2 and SPRK-3 are fourth and sixth order accurate in time.  Then, in Figures  \ref{ZSEM2} and  \ref{ZSEM3}, we show  the relative residuals on the  mass and energy calculated  by SPRK-2 and SPRK-3 on the time interval $[0,200]$, respectively, where we choose the Fourier node 1024 and the time step $\tau = 1/20$, respectively. It can be observed that SPRK-2 and SPRK-3 can conserve the mass and Hamiltonian energy of the classical ZS exactly.

To demonstrate the advantages of the proposed schemes, we fix the Fourier node 4096 and  then investigate
the numerical error accumulations and robustness of the proposed schemes in long numerical simulation as the large time step $\tau = 1/20$ is chosen. The results are summarized in Figures \ref{ComErr1} and \ref{ComErr2}. In Figure \ref{ComErr1}, we can observe that compared with the TS-EWI-FP method \cite{ZhangG21}, the errors provided by the SPRK-2 and SPRK-3 are not only much smaller at the same time steps, but also have a good numerical performance in long time simulations. In Figure \ref{ComErr2}, it is clear to see that the computational results provided by the TS-EWI-FP method is unstable and our schemes are more robust.
%
\begin{figure}[H]
\begin{minipage}[t]{0.5\linewidth}
\centering
\includegraphics[height=5.0cm,width=6.5cm]{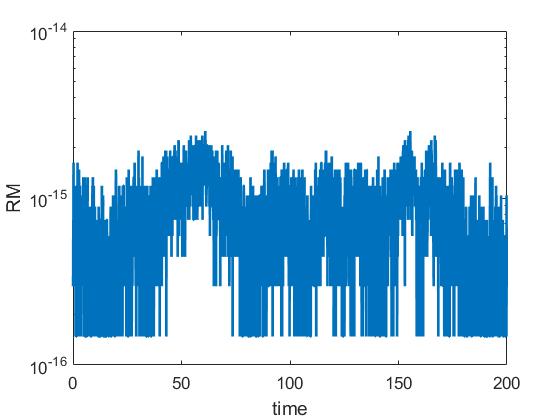}
\end{minipage}%
\begin{minipage}[t]{0.5\linewidth}
\centering
\includegraphics[height=5.0cm,width=6.5cm]{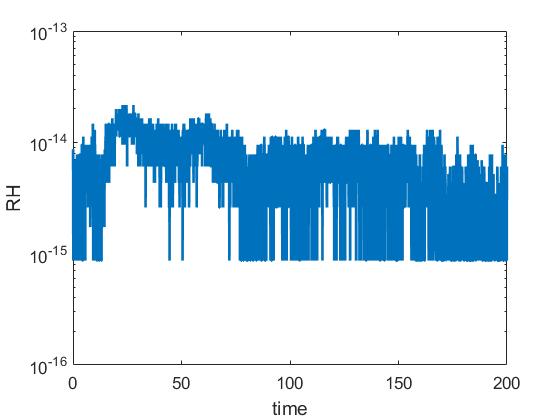}
\end{minipage}
\vspace{-6mm}
\caption{The relative residuals on the mass (left) and energy (right) of SPRK-$2$ for the classical ZS with the initial conditions \eqref{eq5.3}, the Fourier node 1024 and the time step $\tau = 1/20$.}\label{ZSEM2}
\end{figure}


%
\begin{figure}[H]
\begin{minipage}[t]{0.5\linewidth}
\centering
\includegraphics[height=5.0cm,width=6.5cm]{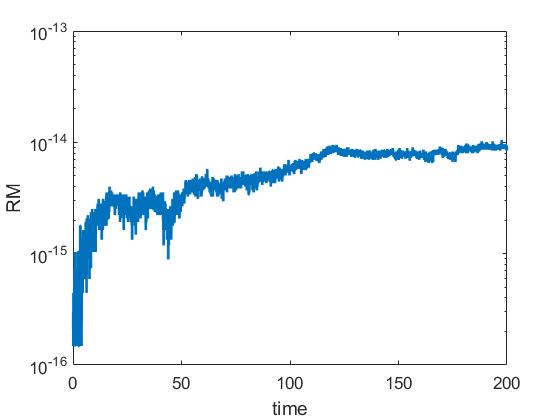}
\end{minipage}%
\begin{minipage}[t]{0.5\linewidth}
\centering
\includegraphics[height=5.0cm,width=6.5cm]{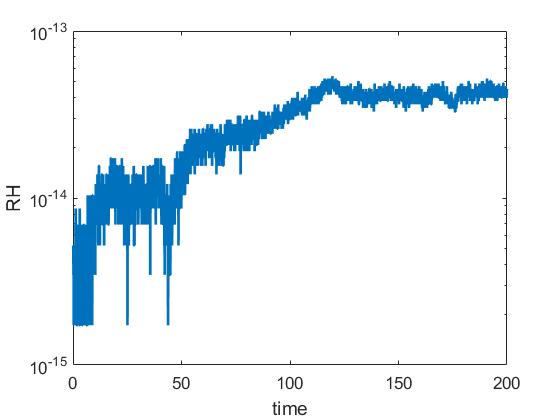}
\end{minipage}
\vspace{-6mm}
\caption{The relative residuals on the mass (left) and  energy (right) of SPRK-3 for the classical ZS with the initial conditions \eqref{eq5.3}, the Fourier node 1024 and the time step $\tau = 1/20$.}\label{ZSEM3}
\end{figure}

\begin{figure}[H]
\subfigure[Errors of $E$  using TS-EWI-FP method]{
\begin{minipage}[t]{0.5\linewidth}
\centering
\includegraphics[height=5.0cm,width=6.5cm]{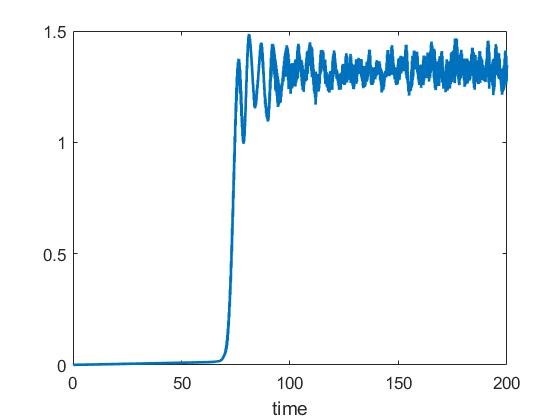}
\end{minipage}%
}%
\subfigure[Errors of $N$ using  TS-EWI-FP method]{
\begin{minipage}[t]{0.5\linewidth}
\centering
\includegraphics[height=5.0cm,width=6.5cm]{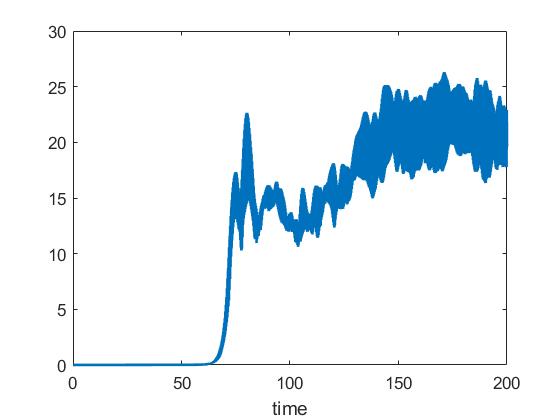}
\end{minipage}%
}%
\end{figure}
 \begin{figure}[H]
\subfigure[Errors of $E$  using SPRK-2]{
\begin{minipage}[t]{0.5\linewidth}
\centering
\includegraphics[height=5.0cm,width=6.5cm]{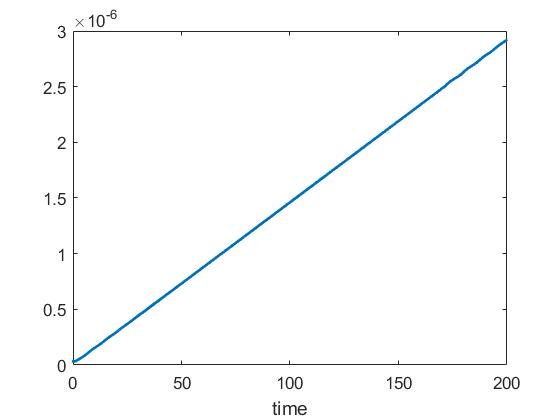}
\end{minipage}
}%
\subfigure[Errors of $N$ using  SPRK-2]{
\begin{minipage}[t]{0.5\linewidth}
\centering
\includegraphics[height=5.0cm,width=6.5cm]{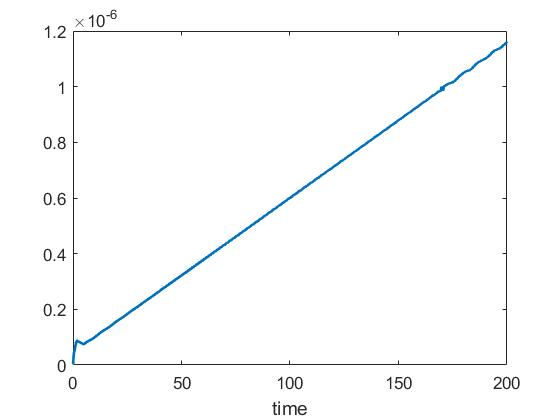}
\end{minipage}
}%
\end{figure}

\begin{figure}[H]
\subfigure[Errors of $E$  using  SPRK-3]{
\begin{minipage}[t]{0.5\linewidth}
\centering
\includegraphics[height=5.0cm,width=6.5cm]{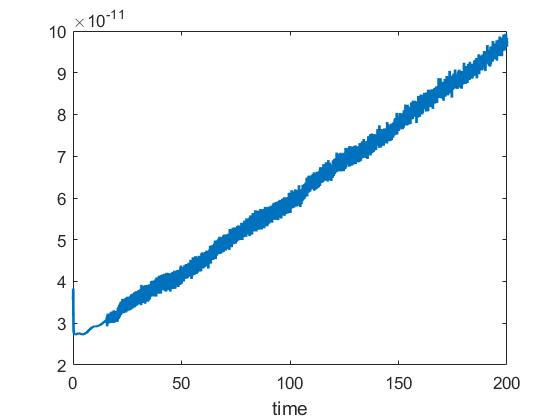}
\end{minipage}
}%
\subfigure[Errors of $N$ using  SPRK-3]{
\begin{minipage}[t]{0.5\linewidth}
\centering
\includegraphics[height=5.0cm,width=6.5cm]{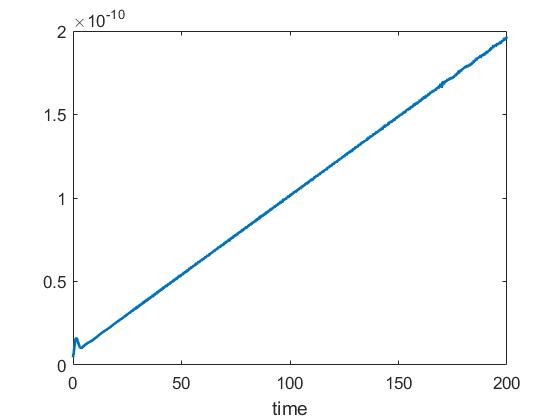}
\end{minipage}
}%
\caption{Time evolution of the numerical errors for the classical ZS with the initial conditions \eqref{eq5.3}, the Fourier node 4096 and the time step $\tau=1/20$.}\label{ComErr1}
\end{figure}


\begin{figure}[H]
\subfigure[]{
\begin{minipage}[t]{0.5\linewidth}
\centering
\includegraphics[height=5.0cm,width=6.5cm]{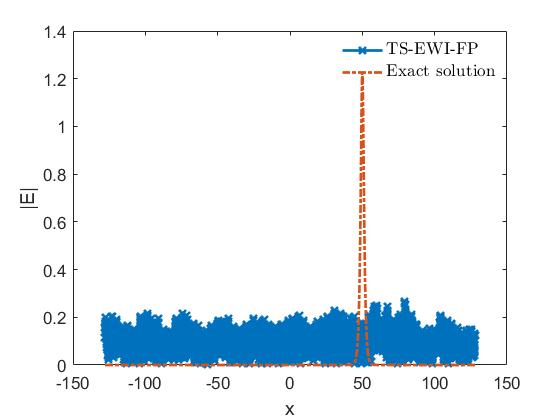}
\end{minipage}%
}%
\subfigure[]{
\begin{minipage}[t]{0.5\linewidth}
\centering
\includegraphics[height=5.0cm,width=6.5cm]{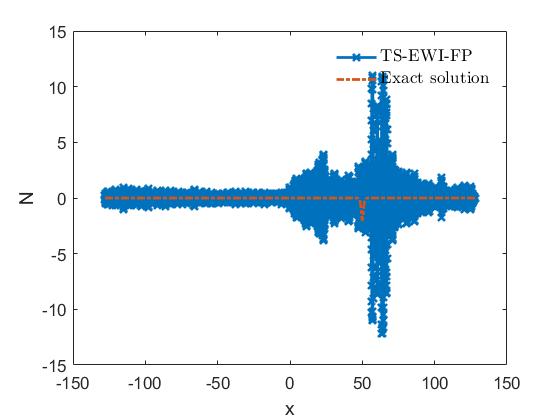}
\end{minipage}%
}%

\subfigure[]{
\begin{minipage}[t]{0.5\linewidth}
\centering
\includegraphics[height=5.0cm,width=6.5cm]{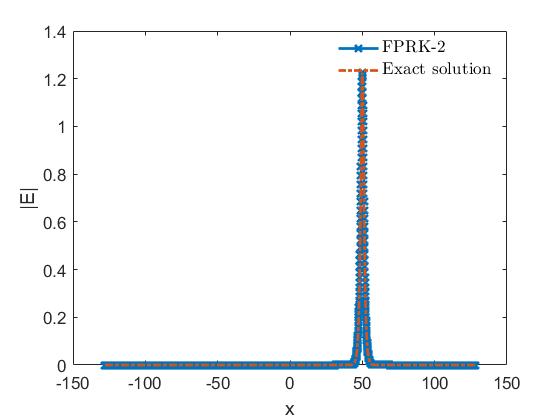}
\end{minipage}
}%
\subfigure[]{
\begin{minipage}[t]{0.5\linewidth}
\centering
\includegraphics[height=5.0cm,width=6.5cm]{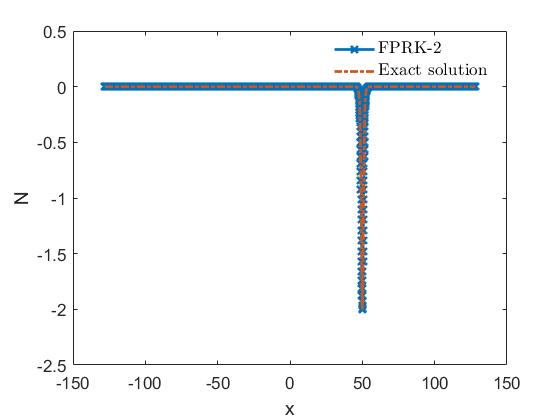}
\end{minipage}
}%
\end{figure}
\begin{figure}[H]
\subfigure[]{
\begin{minipage}[t]{0.5\linewidth}
\centering
\includegraphics[height=5.0cm,width=6.5cm]{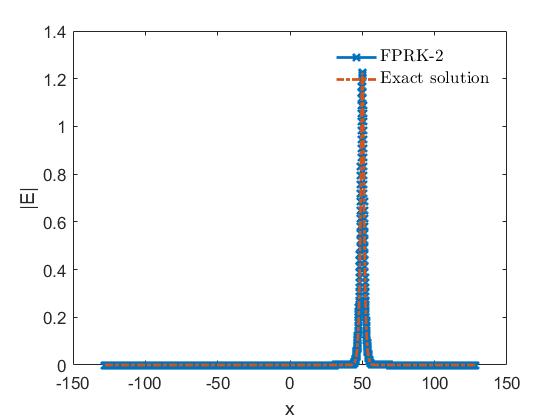}
\end{minipage}
}%
\subfigure[]{
\begin{minipage}[t]{0.5\linewidth}
\centering
\includegraphics[height=5.0cm,width=6.5cm]{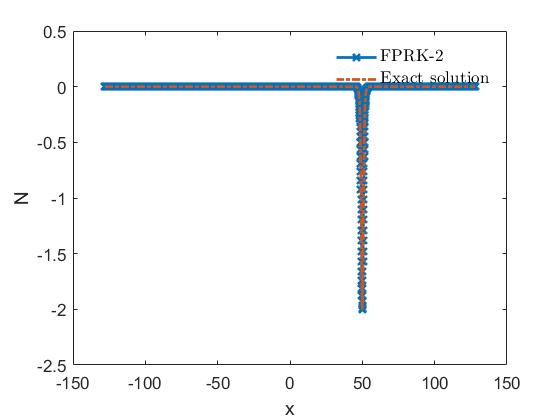}
\end{minipage}
}%
\caption{ The numerical solutions at $T=100$  for the classical ZS with the initial conditions \eqref{eq5.3}, the Fourier node 4096 and the time step $\tau=1/20$.}\label{ComErr2}
\end{figure}

\begin{figure}[H]
\subfigure[SPRK2]{
\begin{minipage}[t]{0.5\linewidth}
\centering
\includegraphics[height=5.0cm,width=6.5cm]{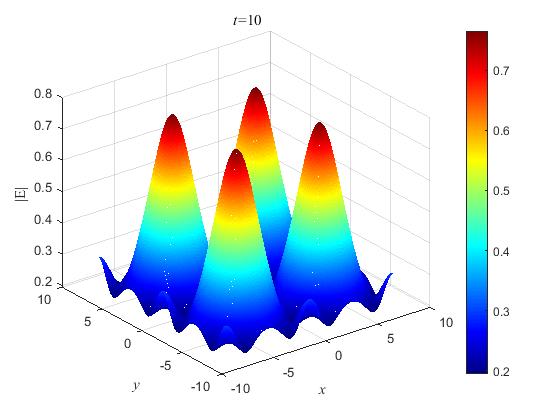}
\end{minipage}%
}%
\subfigure[SPRK2]{
\begin{minipage}[t]{0.5\linewidth}
\centering
\includegraphics[height=5.0cm,width=6.5cm]{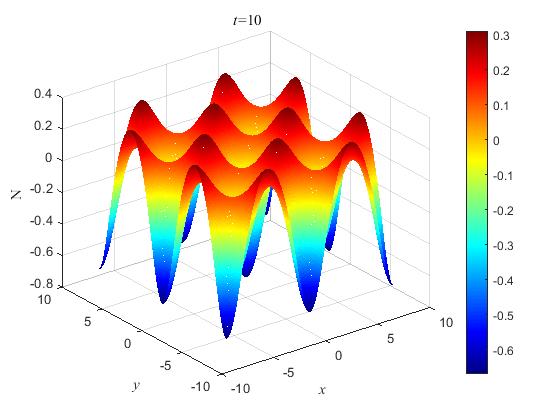}
\end{minipage}%
}%

\subfigure[SPRK3]{
\begin{minipage}[t]{0.5\linewidth}
\centering
\includegraphics[height=5.0cm,width=6.5cm]{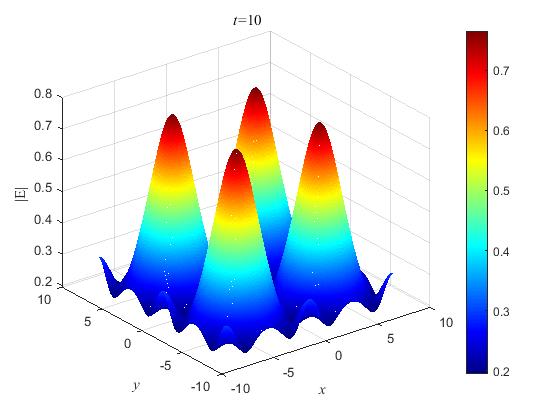}
\end{minipage}
}%
\subfigure[SPRK3]{
\begin{minipage}[t]{0.5\linewidth}
\centering
\includegraphics[height=5.0cm,width=6.5cm]{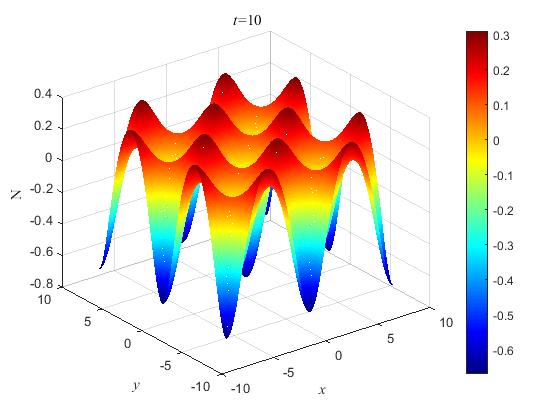}
\end{minipage}
}%
\caption{ The numerical solutions at $T = 10$  for the QZS \eqref{eq1.1} with the initial conditions \eqref{eq5.4}, the parameter $\varepsilon=\frac{1}{2^3}$, the Fourier node $256\times 256 $ and the time step $\tau=1/20$.}\label{2dSPRK}
\end{figure}

\bigskip
{ \bf Example 5.2}.
Consider the two-dimensional QZS  (\ref{eq1.1}) with initial conditions
\begin{equation}
\label{eq5.4}
E_0(x, y)= \cos^2 \frac{ \pi x}{8} \cos^2 \frac{ \pi y}{8} ,\
N_0(x, y)=0,\
N_1(x, y)= 0 ,
\end{equation}
and the periodic boundary condition.

We set the computational domain $\Omega= [-8,8)\times [-8,8)$. Since the exact solution is not known, we take the numerical solution obtained by the proposed SPRK-3
with the Fourier node $256\times 256 $ and the time step $\tau= 10^{-3}$ as the ``reference solution".
Tables \ref{tabt2} and \ref{tabt3} list the temporal errors of SPRK-2 and SPRK-3 at $T = 1$ under different time steps $\tau$ and $\varepsilon$, where the Fourier node is chosen as $256\times 256$. It can be clearly observed that SPRK-2 and SPRK-3 are
really fourth order accurate and sixth order accurate in time,
respectively. Subsequently, the plots for $|E|$ and $N$ at $T=10$ are shown in Figure \ref{2dSPRK}, which implies the numerical solutions obtained by the proposed method are stable in a long time-integration. In Figures  \ref{QZSEM2} and  \ref{QZSEM3}, we display the relative residuals on the mass and energy in time interval $[0,200]$ obtained by SPRK-2 and SPRK-3 with the parameter $\varepsilon=\frac{1}{2^7}$, the Fourier node $256\times 256 $ and the time step $\tau = 1/20$, which implies that the proposed schemes can preserve the mass and Hamiltonian energy of the QZS \eqref{eq1.1} exactly.

To verify the efficiency of the proposed schemes, we also investigated the maximum norm errors and the CPU time
using SPRK-2, SPRK-3 and CNS \cite{CaiYongyong22ANM} at $T=1$ with the parameter $\varepsilon= 1/2^3$ and the Fourier node $256\times 256 $, respectively. The results are summarized in Table  \ref{comparison2}. As illustrated in the Table, we can draw that for a given global error, the proposed schemes spend much less CPU time than CNS, which implies that our schemes are much more efficient. We note that the ``reference solution" is obtained using the proposed SPRK-3 with the parameter $\varepsilon= 1/2^3$, the Fourier node $256\times 256 $ and the time step $\tau= 10^{-3}$, respectively.

Subsequently, by taking the initial conditions \eqref{eq5.4}, the Fourier node $256\times 256 $ and the time step $\tau = 1/100$, we apply SPRK-2 and SPRK-3 to study the convergence rates of the QZS \eqref{eq1.1} to its limiting model, i.e. the classical ZS ($\varepsilon= 0$), respectively.
Figure \ref{conv1} shows the maximum norm errors calculated by using SPRK-2 and SPRK-3 between the QZS \eqref{eq1.1} and the corresponding classical ZS at $ T=10$ with different parameters $\varepsilon = 2^{-(2j+1)}, j=2,3,4,5,6,7$. As illustrated in the figures, the solution  of the QZS \eqref{eq1.1} with the initial conditions \eqref{eq5.4} converges to  the classical ZS quadratically with respect to $\varepsilon $, which is consistent with the existing theoretical result presented in \cite{FangYungFu19b}.


\begin{table}[H]
\caption{Temporal errors provided by SPRK-2 at $T=1$ for the QZS \eqref{eq1.1} with \eqref{eq5.4} and different $\varepsilon$.}\label{tabt2}
\setlength{\tabcolsep}{3mm}
\centering
\begin{tabular}{cccccccc} \toprule
$e_{\varepsilon}(1)$& $\tau_0 = 1/4 $   & $\tau_0/2 $ & $\tau_0/2^2 $  & $\tau_0  /2^3 $ & $\tau_0  /2^4 $   \\
\midrule
$\varepsilon =\frac{1}{2^2}$   & 7.05e-5  & 4.99e-6   &   3.32e-7    &  2.11e-8    &  1.32e-9    \\
  rate      & -   & 3.82 &3.91  & 3.98  &    4.00   \\
\midrule
$\varepsilon =\frac{1}{2^3}$   & 6.03e-5  & 4.50e-6    &  2.93e-7    &  1.85e-8    &  1.16e-9    \\
                      rate      & -  & 3.75 & 3.94  & 3.99  &    4.00    \\
\midrule
$\varepsilon =\frac{1}{2^4}$   & 4.67e-5   &   3.45e-6    &  2.24e-7   &   1.41e-8    &  8.83e-10   \\
                  rate      & -   & 3.76 & 3.95  & 3.99  &    4.00  \\
\midrule
$\varepsilon =\frac{1}{2^{5}}$   & 4.30e-5 &  3.16e-6 &  2.04e-7  & 1.29e-8 &  8.07e-10  \\
                  rate      & -  & 3.77 &   3.95  & 3.99  &    4.00   \\
\midrule
$\varepsilon =\frac{1}{2^6}$   &4.20e-5  & 3.08e-6  & 2.00e-7  & 1.26e-8  & 7.88e-10  \\
                  rate      & -  & 3.77 &3.95  & 3.99  &   4.00     \\
\midrule
$\varepsilon =\frac{1}{2^{7}}$  & 1.46e-5 &  3.07e-6 &  1.98e-7  & 1.25e-8   &7.83e-10   \\
                  rate      & -    & 3.77 &3.95  &    3.99   &    4.00     \\
\midrule \midrule
$n_{\varepsilon}(1)$ & $\tau_0 = 1/4 $   & $\tau_0/2 $ & $\tau_0/2^2 $  & $\tau_0  /2^3 $ & $\tau_0  /2^4 $  \\
\midrule
$\varepsilon =\frac{1}{2^2}$     & 2.30e-5   & 1.51e-6    &  9.59e-8   &   6.00e-9   &   3.75e-10    \\
   rate      & -    &  3.92 & 3.98  &    4.00  &    4.00   \\
\midrule
$\varepsilon =\frac{1}{2^3}$   & 9.79e-6  &  6.20e-7   &   3.89e-8   &   2.43e-9   &   1.52e-10    \\
                   rate      & -   &  3.98 &  4.00    &    4.00     &   4.00       \\
\midrule
$\varepsilon =\frac{1}{2^4}$   &  9.33e-6  &  6.00e-07   &  3.79e-8  &  2.37e-9 &   1.49e-10   \\
                  rate      & -   &  3.96 & 3.99    &    4.00     &   4.00     \\
\midrule
$\varepsilon =\frac{1}{2^{5}}$  &1.01e-5 &  6.57e-7 &  4.15e-8 &  2.60e-9  & 1.63e-10  \\
                  rate      & -   &  3.95 &  3.98    &    4.00     &   4.00      \\
\midrule
$\varepsilon =\frac{1}{2^6}$  &1.03e-5  & 6.71e-7 &  4.24e-8 &  2.66e-9  & 1.66e-10   \\
                  rate      & -    &  3.95 &  3.98     &   4.00    &   4.00       \\
\midrule
$\varepsilon =\frac{1}{2^{7}}$  &1.04e-5 &  6.74e-7 &  4.27e-8 &  2.68e-9 &  1.67e-10   \\
                  rate      & -   &  3.95 &  3.98    &   4.00    &    4.00         \\
\bottomrule
  \end{tabular}
\end{table}


\begin{table}[H]
\caption{Temporal errors provided by SPRK-3 at $T=1$ for the QZS \eqref{eq1.1} with \eqref{eq5.4} and different $\varepsilon$.}\label{tabt3}
\setlength{\tabcolsep}{3mm}
\centering
\begin{tabular}{cccccccc} \toprule
$e_{\varepsilon}(1)$   & $\tau_0 = 1/4 $ & $\tau_0/2 $  & $\tau_0  /2^2 $ & $\tau_0  /2^3 $   \\
\midrule
$\varepsilon =\frac{1}{2^2}$  &2.95e-6  &  6.41e-8   & 1.10e-9   & 1.76e-11    \\
  rate      & -   & 5.53    &   5.87    &   5.97     \\
\midrule
$\varepsilon =\frac{1}{2^3}$   &1.63e-6   & 2.95e-8   & 4.78e-10  &  7.54e-12    \\
                      rate      & -   &  5.79     &  5.95    &  5.99   \\
\midrule
$\varepsilon =\frac{1}{2^4}$  &1.19e-6  &  2.13e-8   & 3.44e-10   & 5.42e-12    \\
                  rate      & -   & 5.80  &    5.95   &   5.99       \\
\midrule
$\varepsilon =\frac{1}{2^{5}}$  & 1.08e-6  &  1.91e-8  &  3.08e-10   & 4.86e-12  \\
                  rate      & -   &  5.82  &   5.95    &  5.99     \\
\midrule
$\varepsilon =\frac{1}{2^6}$  & 1.05e-6  &  1.86e-8   & 3.00e-10  &  4.72e-12   \\
                  rate      & -   &5.82    &  5.96   &  5.99      \\
\midrule
$\varepsilon =\frac{1}{2^{7}}$  &1.04e-6  &  1.85e-8  &  2.98e-10   & 4.69e-12   \\
                  rate      & -   &  5.82    &  5.96   &   5.99     \\
\midrule \midrule
$n_{\varepsilon}(1)$   & $\tau_0 =1/4 $  & $\tau_0/2 $  & $\tau_0  /2^2 $ & $\tau_0  /2^3 $  \\
\midrule
$\varepsilon =\frac{1}{2^2}$     &2.70e-7  &  5.10e-9   & 8.41e-11  &  1.33e-12     \\
   rate      & -   & 5.73   &  5.92   &    5.98   \\
\midrule
$\varepsilon =\frac{1}{2^3}$  &6.89e-8  &  1.18e-9  &  1.89e-11  &  3.01e-13  \\
                   rate      & -   &  5.86  &   5.96    &  5.99      \\
\midrule
$\varepsilon =\frac{1}{2^4}$  &1.24e-7 &   2.08e-9   & 3.31e-11   &  5.33e-13   \\
                  rate      & -   & 5.90 &   5.97  &   5.98   \\
\midrule
$\varepsilon =\frac{1}{2^{5}}$  &1.56e-7  &  2.60e-9  &  4.14e-11  &  6.91e-13  \\
                  rate      & -   &   5.90 &   5.97  &   5.99     \\
\midrule
$\varepsilon =\frac{1}{2^6}$  &1.63e-7  &  2.72e-9   & 4.32e-11   &  8.00e-13    \\
                  rate      & -   &    5.90   &  5.98  &   5.98    \\
\midrule
$\varepsilon =\frac{1}{2^{7}}$  &1.65e-7   & 2.75e-9   & 4.37e-11  &  7.29e-13     \\
                  rate      & -   &  5.90   &  5.98  &   5.98   \\
\bottomrule
  \end{tabular}
\end{table}

\begin{figure}[H]
\begin{minipage}[t]{0.5\linewidth}
\centering
\includegraphics[height=5.0cm,width=6.5cm]{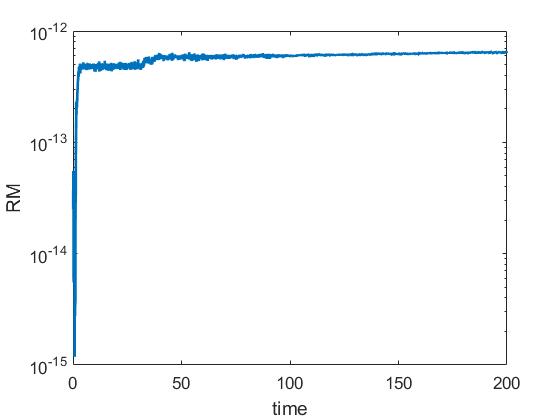}
\end{minipage}%
\begin{minipage}[t]{0.5\linewidth}
\centering
\includegraphics[height=5.0cm,width=6.5cm]{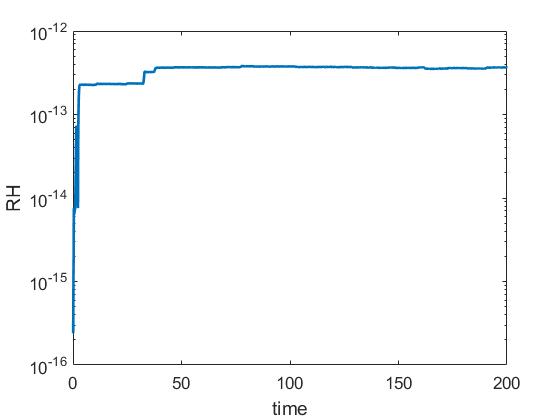}
\end{minipage}
\vspace{-6mm}
\caption{The relative residuals on the mass (left) and energy (right) of SPRK-2 for the QZS \eqref{eq1.1} with the initial conditions \eqref{eq5.4}, the parameter $\varepsilon=\frac{1}{2^7}$, the Fourier node $256\times 256 $ and the time step $\tau = 1/20$.}\label{QZSEM2}
\end{figure}

%
\begin{figure}[H]
\begin{minipage}[t]{0.5\linewidth}
\centering
\includegraphics[height=5.0cm,width=6.5cm]{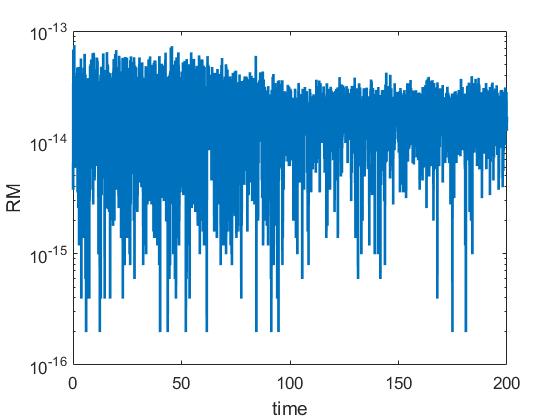}
\end{minipage}%
\begin{minipage}[t]{0.5\linewidth}
\centering
\includegraphics[height=5.0cm,width=6.5cm]{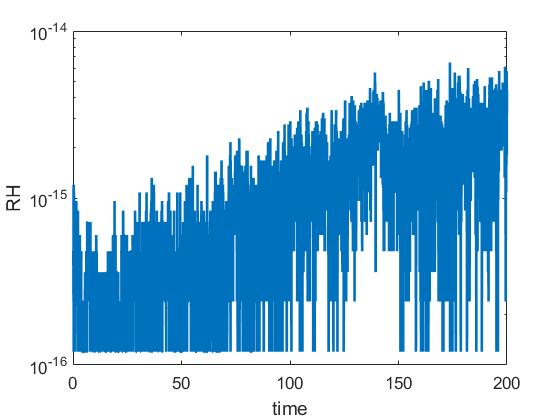}
\end{minipage}
\vspace{-6mm}
\caption{The relative residuals on the mass (left) and energy (right) of SPRK-3 for the QZS \eqref{eq1.1} with the initial conditions \eqref{eq5.4}, the parameter $\varepsilon=\frac{1}{2^7}$, the Fourier node $256\times 256 $ and the time step $\tau = 1/20$.}\label{QZSEM3}
\end{figure}

\begin{table}[H]
\caption{Numerical errors and computational CPU times using three numerical
schemes solving the QZS \eqref{eq1.1} with the initial conditions \eqref{eq5.4} at $T=1$, the parameter $\varepsilon= 1/2^3$ and the Fourier node $256\times 256 $.}\label{comparison2}
\setlength{\tabcolsep}{3mm}
\centering
 \begin{tabular}{ccccccc} \toprule
             Scheme      &$\tau$      & $e_{\varepsilon}(1)$ & $n_{\varepsilon}(1)$ & CPU time (s)     \\[1ex]
\midrule
             CNS \cite{CaiYongyong22ANM}         & $1.0\times10^{-5}$&7.18e-10    &  1.22e-09 &   224.19&\\[1ex]
\midrule
             SPRK-2      & $1.0\times10^{-2}$& 1.94e-10    &   2.55e-11 &   5.27&\\[1ex]
\midrule
             SPRK-3     & $5.0\times10^{-2}$  & 1.26e-10    &   4.99e-12 &   2.68&\\[1ex]

\bottomrule
  \end{tabular}
\end{table}

%
\begin{figure}[H]
\begin{minipage}[t]{0.5\linewidth}
\centering
\includegraphics[height=5.0cm,width=6.5cm]{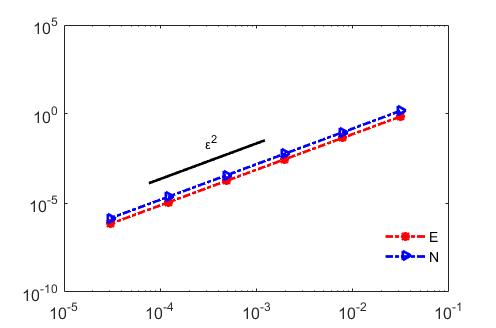}
\end{minipage}%
\begin{minipage}[t]{0.5\linewidth}
\centering
\includegraphics[height=5.0cm,width=6.5cm]{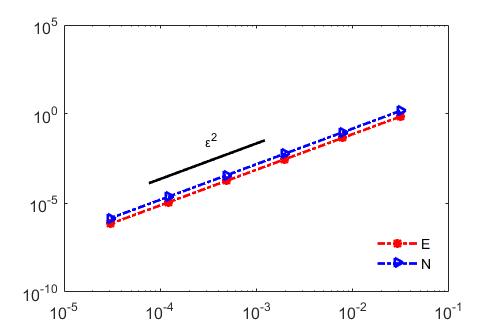}
\end{minipage}
\vspace{-6mm}
\caption{Convergence rates of SPRK-2 (left) and SPRK-3 (right) between the QZS \eqref{eq1.1} and the classical ZS at $T=10$, respectively, where we choose the initial conditions \eqref{eq5.4}, the Fourier node $256\times 256 $ and the time step $\tau = 1/100$, respectively.}\label{conv1}
\end{figure}

\subsection{Dynamic simulation of one dimensional QZS}

{ \bf Example 5.3}. In this subsection, we choose the initial conditions as (cf. \cite{ZhangG21,ZhangSu21})
\begin{equation}
\label{ex3}
\begin{split}
&E_0(x)={\rm i} \sum\limits_{j=1}^2\sqrt{2\left(1-V_j^{2}\right)} \operatorname{sech}\left(x-x_{j}\right) e^{{\rm i}V_j\left(x-x_j\right)/2},\\
&N_0(x)=-2 \sum\limits_{j=1}^2 \operatorname{sech}^{2}\left(x-x_{j}\right),\\
&N_1(x)=-4\sum\limits_{j=1}^2 V_j \operatorname{sech}^{2}\left(x-x_{j}\right)
\operatorname{tanh}\left(x-x_{j}\right),\ x\in\Omega,
\end{split}
\end{equation}
where $x_1$ and $x_2$ are initial locations of the two solitary waves, and
$V_1$ and $V_2$ are the propagation velocity and the sign  means moving  to the right or left.
For brevity, all computations are performed by using SPRK-2 with the Fourier node 4000 and the time step $\tau=1/20$, where the computational domain is chosen as $\Omega=
[-200, 200)$. Moreover, we take the following parameters as:
 \begin{itemize}
\item  Case I: $x_1=-x_2=-30,\ V_1=-V_2=\frac{1}{2}$;
\item Case II: $x_1=-x_2=-30,\ V_1=\frac{3}{4},\ V_2=-\frac{1}{2}$;
\item  Case III: $x_1=-x_2=-5,\ V_1=\frac{3}{4},\ V_2=-\frac{1}{2}$.
 \end{itemize}
\begin{figure}[H]
\centering

\subfigure[$\varepsilon =0$]{
\begin{minipage}[t]{0.5\linewidth}
\centering
\includegraphics[height=5cm,width=6.5cm]{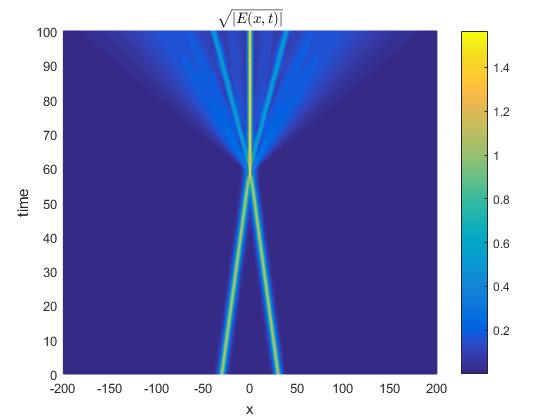}
\end{minipage}%
}%
\subfigure[$\varepsilon =0$]{
\begin{minipage}[t]{0.5\linewidth}
\centering
\includegraphics[height=5cm,width=6.5cm]{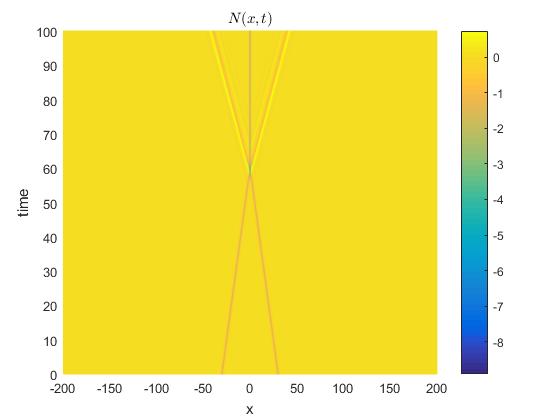}
\end{minipage}%
}%
 \\
\subfigure[$\varepsilon =\frac{1}{2^2}$]{
\begin{minipage}[t]{0.5\linewidth}
\centering
\includegraphics[height=5cm,width=6.5cm]{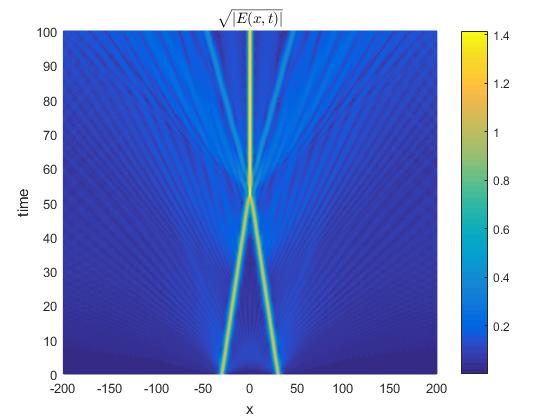}
\end{minipage}
}%
\subfigure[$\varepsilon =\frac{1}{2^2}$]{
\begin{minipage}[t]{0.5\linewidth}
\centering
\includegraphics[height=5cm,width=6.5cm]{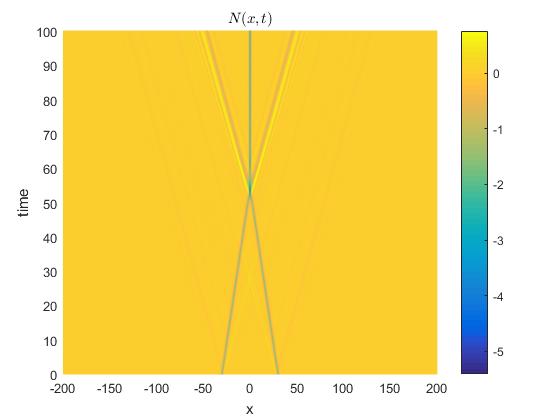}
\end{minipage}
}%
\\
\subfigure[$\varepsilon =1 $]{
\begin{minipage}[t]{0.5\linewidth}
\centering
\includegraphics[height=5cm,width=6.5cm]{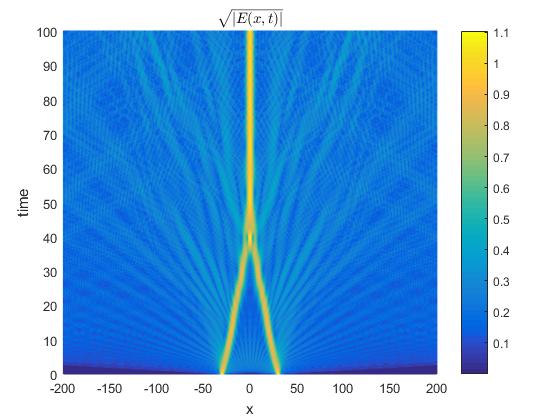}
\end{minipage}
}%
\subfigure[$\varepsilon =1 $]{
\begin{minipage}[t]{0.5\linewidth}
\centering
\includegraphics[height=5cm,width=6.5cm]{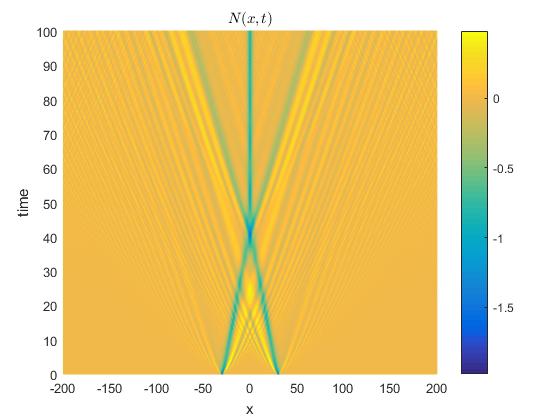}
\end{minipage}
}%
\centering
\caption{Inelastic collision between two solitons for the QZS \eqref{eq1.1} with the initial conditions \eqref{ex3} under case I.}\label{interac1}
\end{figure}

The surface plots of the interaction of two solitary waves for QZS ($\varepsilon \neq 0$) \eqref{eq1.1}
and the classical ZS ($\varepsilon=0$)  under cases (I)-(III)  are
demonstrated in Figures \ref{interac1}-\ref{interac3},
respectively, which implies that: (i) all the collisions between two solitons are not elastic; (ii) when two initially well-separated solitons with opposite  propagation velocities (cf. case (I) in Figure \ref{interac1})
or different propagation speeds  (cf. case (II) in Figure \ref{interac2}), they  collide and fuse into a new soliton with the strengthened amplitude and the narrower width; (iii) the amplitude-weakened solitons with propagation speeds changed and some small radiation are generated during the collision; (iv) when the  initial locations are not initially well-separated (cf. case (III) in Figure \ref{interac3}), the dynamics is considerably  more complicated, it seems that a periodic perturbation on the position of some localized pulse. In addition, from Figures \ref{interac1}-\ref{interac3}, we also observe that the soliton-soliton collisions of the QZS are more unstable
than the corresponding classical ZS after collision, and the quantum effect makes the chaos much more obvious. The larger the quantum effect is, the more obvious the spatiotemporal chaos is. In particular, for the  strong quantum regime  $\varepsilon=1$,  there are small outgoing waves emitting before colliding,
and the chaos is much more obvious than the classical one. In Figures \ref{CONinterac2} and \ref{CONinterac1} , we display the relative residuals on mass and energy provided by SPRK-2, which shows that the proposed scheme can conserve the discrete mass and energy exactly.
\begin{figure}[H]
\centering

\subfigure[$\varepsilon =0$]{
\begin{minipage}[t]{0.5\linewidth}
\centering
\includegraphics[height=5cm,width=6.5cm]{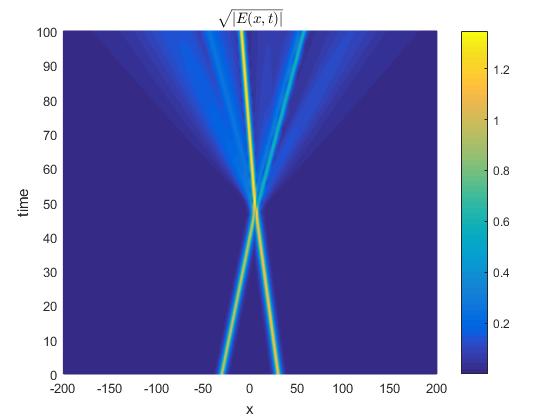}
\end{minipage}%
}%
\subfigure[$\varepsilon =0$]{
\begin{minipage}[t]{0.5\linewidth}
\centering
\includegraphics[height=5cm,width=6.5cm]{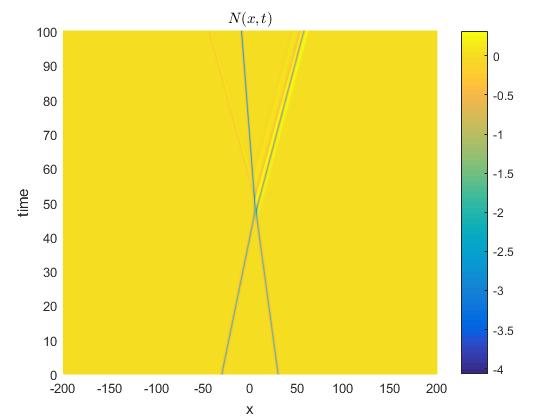}
\end{minipage}%
}%
 \\
\subfigure[$\varepsilon =\frac{1}{2^2}$]{
\begin{minipage}[t]{0.5\linewidth}
\centering
\includegraphics[height=5cm,width=6.5cm]{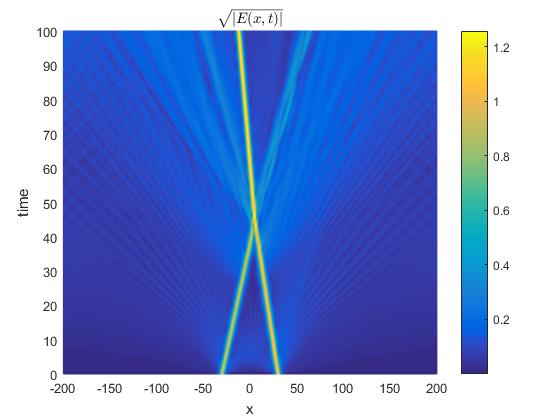}
\end{minipage}
}%
\subfigure[$\varepsilon =\frac{1}{2^2}$]{
\begin{minipage}[t]{0.5\linewidth}
\centering
\includegraphics[height=5cm,width=6.5cm]{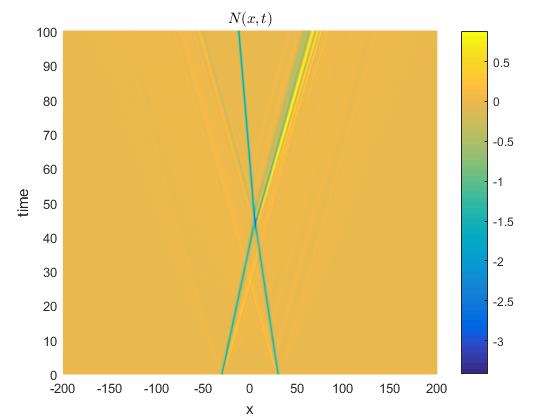}
\end{minipage}
}%
\\
\subfigure[$\varepsilon =1 $]{
\begin{minipage}[t]{0.5\linewidth}
\centering
\includegraphics[height=5cm,width=6.5cm]{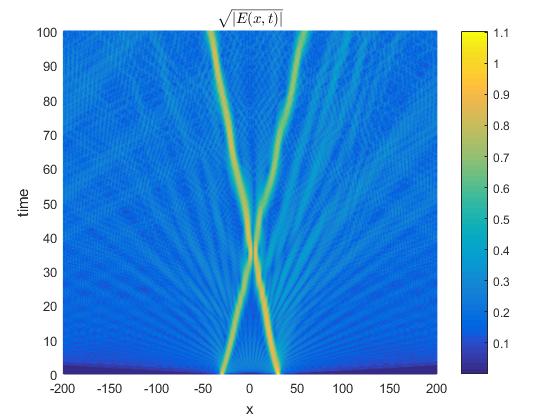}
\end{minipage}
}%
\subfigure[$\varepsilon =1 $]{
\begin{minipage}[t]{0.5\linewidth}
\centering
\includegraphics[height=5cm,width=6.5cm]{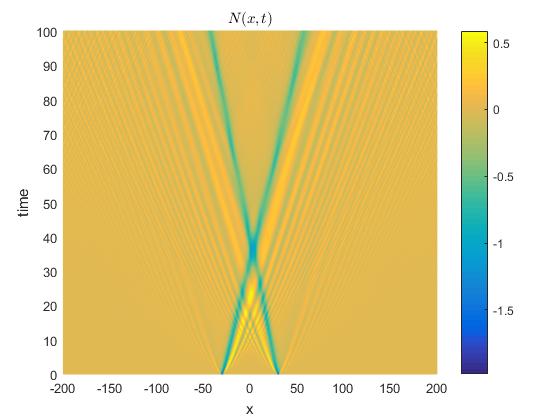}
\end{minipage}
}%
\centering
\caption{Inelastic collision between two solitons for the QZS \eqref{eq1.1} with the initial conditions \eqref{ex3} under case II.}\label{interac2}
\end{figure}
\begin{figure}[H]
\centering

\subfigure[$\varepsilon =0$]{
\begin{minipage}[t]{0.5\linewidth}
\centering
\includegraphics[height=5cm,width=6.5cm]{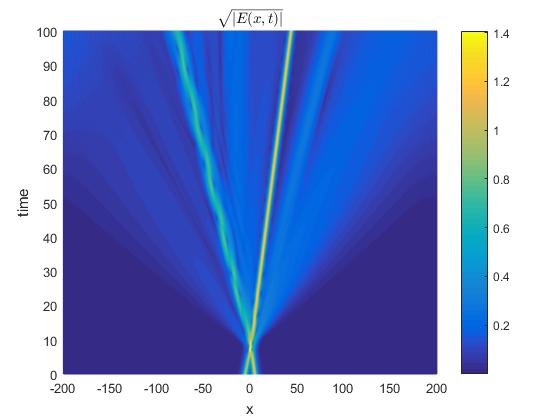}
\end{minipage}%
}%
\subfigure[$\varepsilon =0$]{
\begin{minipage}[t]{0.5\linewidth}
\centering
\includegraphics[height=5cm,width=6.5cm]{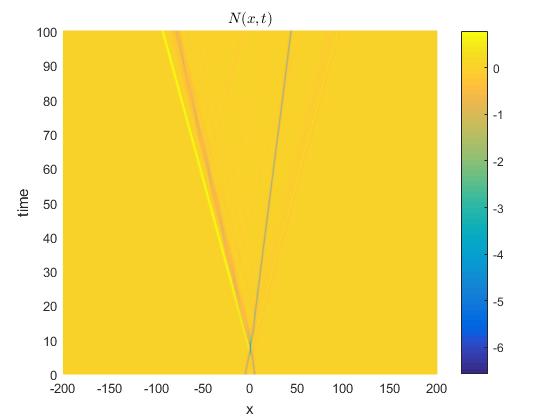}
\end{minipage}%
}%
 \\
\subfigure[$\varepsilon =\frac{1}{2^2}$]{
\begin{minipage}[t]{0.5\linewidth}
\centering
\includegraphics[height=5cm,width=6.5cm]{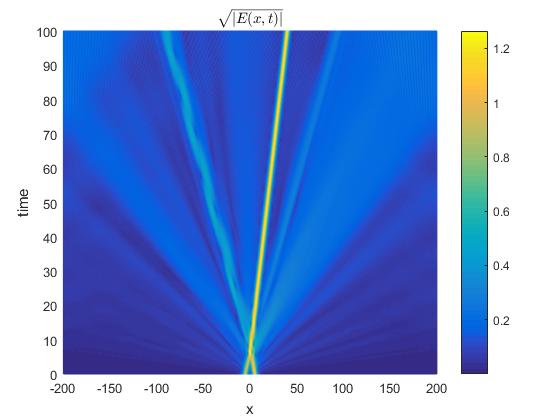}
\end{minipage}
}%
\subfigure[$\varepsilon =\frac{1}{2^2}$]{
\begin{minipage}[t]{0.5\linewidth}
\centering
\includegraphics[height=5cm,width=6.5cm]{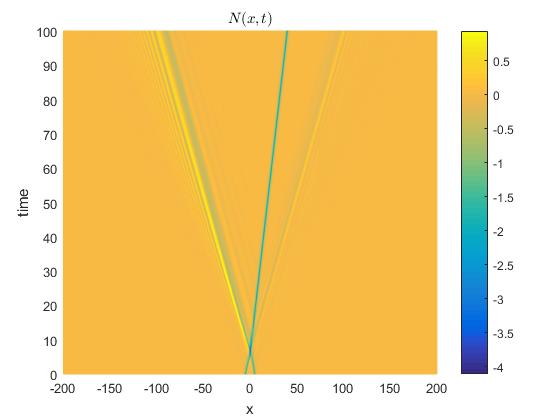}
\end{minipage}
}%
\\
\subfigure[$\varepsilon =1 $]{
\begin{minipage}[t]{0.5\linewidth}
\centering
\includegraphics[height=5cm,width=6.5cm]{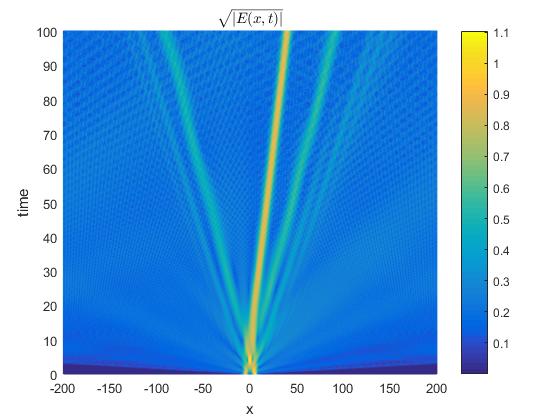}
\end{minipage}
}%
\subfigure[$\varepsilon =1 $]{
\begin{minipage}[t]{0.5\linewidth}
\centering
\includegraphics[height=5cm,width=6.5cm]{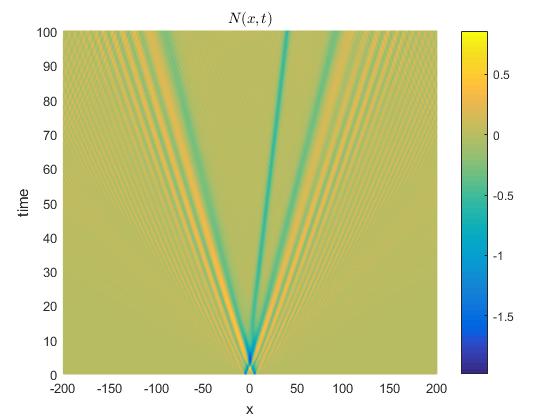}
\end{minipage}
}%

\centering
\caption{Inelastic collision between two solitons  for the QZS \eqref{eq1.1} with the initial conditions \eqref{ex3} under case III.}\label{interac3}
\end{figure}
\begin{figure}[H]
\begin{minipage}[t]{0.33\linewidth}
\centering
\includegraphics[height=5cm,width=5cm]{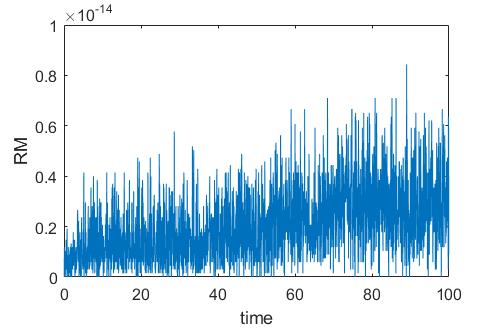}
\end{minipage}%
\begin{minipage}[t]{0.33\linewidth}
\centering
\includegraphics[height=5cm,width=5cm]{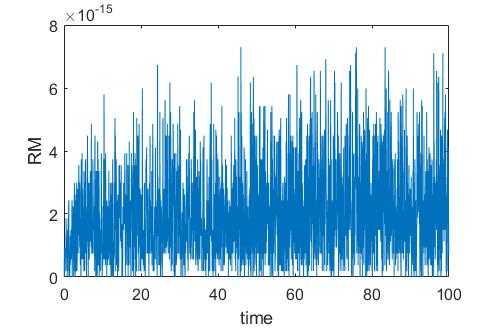}
\end{minipage}
\begin{minipage}[t]{0.33\linewidth}
\centering
\includegraphics[height=5cm,width=5cm]{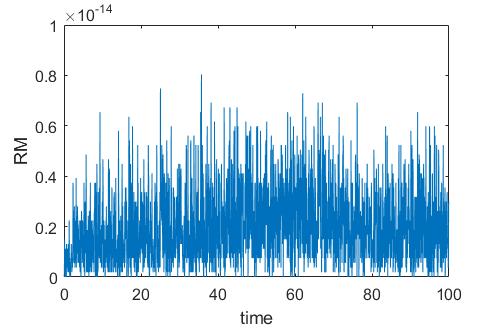}
\end{minipage}%
\vspace{-3mm}
\caption{The relative residuals on the  mass  of SPRK-2 for the QZS \eqref{eq1.1} with the initial conditions \eqref{ex3} under case I, II and III (from  left to right), respectively with the parameter $\varepsilon =\frac{1}{2^2}$, the Fourier node 4000 and the time step $\tau=1/20$.}\label{CONinterac2}
\end{figure}

\begin{figure}[H]
\begin{minipage}[t]{0.33\linewidth}
\centering
\includegraphics[height=5cm,width=5cm]{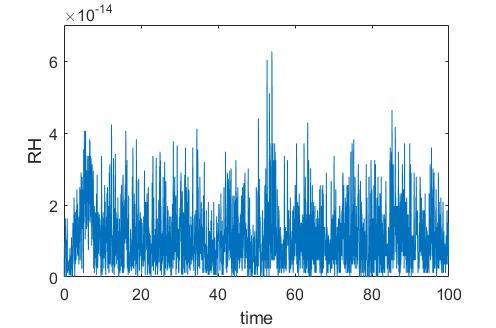}
\end{minipage}%
\begin{minipage}[t]{0.33\linewidth}
\centering
\includegraphics[height=5cm,width=5cm]{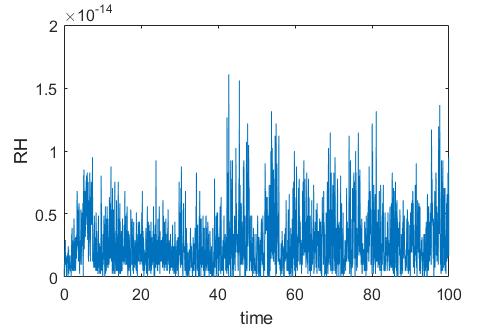}
\end{minipage}
\begin{minipage}[t]{0.33\linewidth}
\centering
\includegraphics[height=5cm,width=5cm]{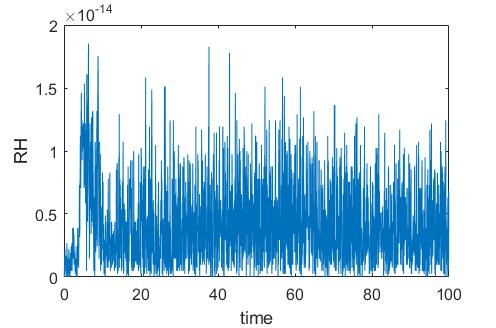}
\end{minipage}%
\vspace{-3mm}
\caption{The relative residuals on the energy  of SPRK-2 for the QZS \eqref{eq1.1} with the initial conditions \eqref{ex3} under case I, II and III (from  left to right), respectively, with the parameter $\varepsilon =\frac{1}{2^2}$, the Fourier node 4000 and the time step $\tau=1/20$.}\label{CONinterac1}
\end{figure}
{ \bf Example 5.4}. The initial  conditions are chosen as (cf. \cite{Misra09})
\begin{equation}
\label{ex4}
\begin{split}
&E_0(x)= E_0  (1 + \beta \cos(kx)) ,\\
&N_0(x)=- \sqrt{2} E_0 k \beta    \cos(kx),\quad
N_1(x)= 0 ,
\end{split}
\end{equation}
where $E_0 = (k / \sqrt{2}) (1 + \varepsilon^2 k^2) $ is the
amplitude of the pump Langmuir wave,  $0 < k < \sqrt{2} E_0 $
and $ \beta $ represents a relatively  small constant to emphasize that the perturbation is  small.

We take the computational domain $\Omega=
[-100, 100)$ and set the parameters as $k = 0.7$, $\beta = 0.001$. For brevity, all computations are performed by using SPRK-2 over the time interval $[0,200]$ with the Fourier node 2000 and the time step $\tau = 1/20$.
Figure \ref{patt1} shows that many solitary patterns can be generated and excited through
the modulational instability of unstable harmonic modes.
It can be clearly observed that the QZS \eqref{eq1.1} is more unstable than  the corresponding  classical ZS.
In particular, for the strong quantum regime $\varepsilon=1$,
numerical simulation also indicates that
the motion of the trains leads to more collision among the neighboring coherent solitary patterns,
 and fuse into the newer pattern with strengthened amplitude.
This space-time evolution reveals that  the spatiotemporal chaos is more obvious as the classical one.
The numerical results are in good agreement with the results given in \cite{Misra09}.
Figure \ref{CONpatt} display the relative residuals on the  mass and energy provided by SPRK-2 with the parameter $\varepsilon =0$, which implies that SPRK-2  preserve the mass and energy conservations  exactly.

\section{Conclusions}

In this paper,  we  develop a novel class of high-order accurate structure-preserving methods for solving the QZS \eqref{eq1.1}, which is based on the idea of the QAV approach, the symplectic RK method together with the standard Fourier pseudo-spectral method in space. We show that the proposed scheme can preserve the discrete mass and original Hamiltonian energy exactly. In addition, an efficient fixed point iteration is presented to solve the resulting nonlinear equations of the proposed scheme. Numerical experiments for the QZS \eqref{eq1.1} are carried out to illustrate the capability and accuracy
of the new schemes. We also use our new scheme to numerically simulate the soliton-soliton interaction and the pattern dynamics of the QZS \eqref{eq1.1} in 1D. We numerically observe that the numerical solution  of the QZS \eqref{eq1.1} converges to the classical Zakharov system quadratically in the semi-classical limit.
As far as we know, there are some works (e.g., see Refs. \cite{CaiYongyong22ANM,ZhangSu21}) on optimal error estimates of second-order mass- and energy-conserving
schemes for the QZS \eqref{eq1.1}, but the error estimate of high-order ones is still not available. Thus, how to establish optimal error
estimates for the proposed methods will be an interesting topic for future studies.


\begin{figure}[H]
\subfigure[$\varepsilon =0$]{
\begin{minipage}[t]{0.5\linewidth}
\centering
\includegraphics[height=5cm,width=6.0cm]{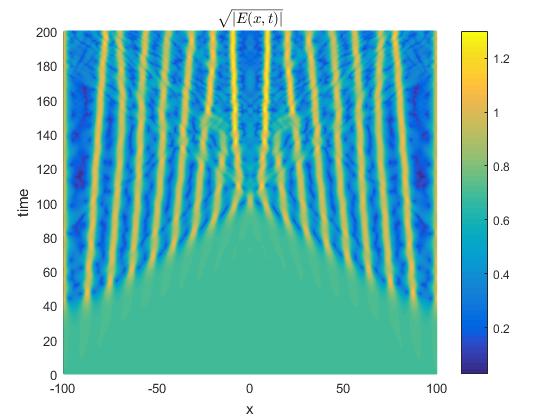}
\end{minipage}%
}%
\subfigure[$\varepsilon =1/2^4$]{
\begin{minipage}[t]{0.5\linewidth}
\centering
\includegraphics[height=5cm,width=6.0cm]{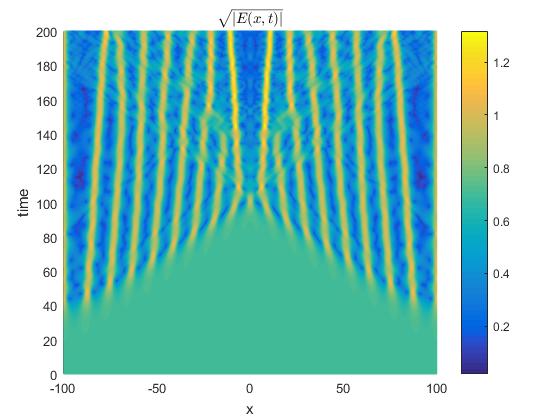}
\end{minipage}%
}%
 \\
\subfigure[$\varepsilon =1/2^2$]{
\begin{minipage}[t]{0.5\linewidth}
\centering
\includegraphics[height=5cm,width=6.0cm]{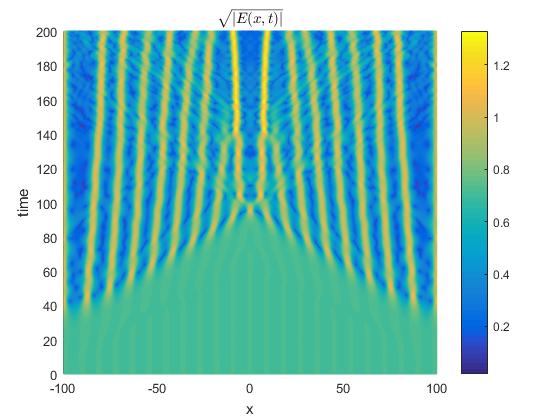}
\end{minipage}
}%
\subfigure[$\varepsilon =1$]{
\begin{minipage}[t]{0.5\linewidth}
\centering
\includegraphics[height=5cm,width=6.0cm]{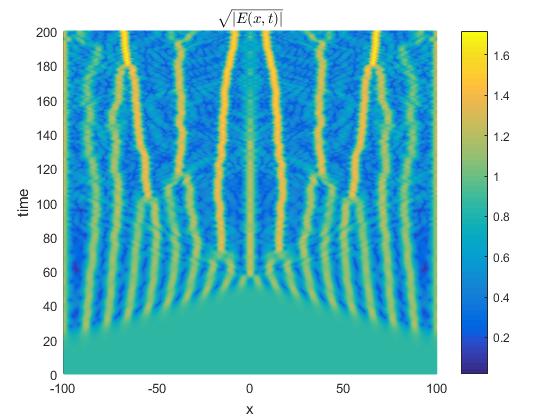}
\end{minipage}
}%
\caption{The contours of $|\sqrt{E(x,t)}|$ for pattern dynamics of SPRK-2 for the QZS \eqref{eq1.1} with the initial conditions \eqref{ex4}.}\label{patt1}
\end{figure}

%
\begin{figure}[H]
\begin{minipage}[t]{0.5\linewidth}
\centering
\includegraphics[height=5cm,width=6.0cm]{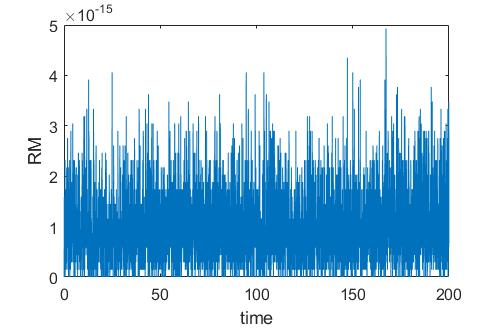}
\end{minipage}%
\begin{minipage}[t]{0.5\linewidth}
\centering
\includegraphics[height=5cm,width=6.0cm]{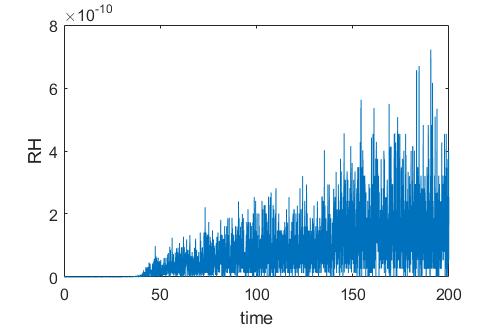}
\end{minipage}
\vspace{-6mm}
\caption{The relative residuals in mass (left) and energy (right) of SPRK-2 for the QZS \eqref{eq1.1}  with the initial conditions \eqref{ex4},  the parameter $\varepsilon =0$, the Fourier node 2000 and the time step $\tau = 1/20$.}\label{CONpatt}
\end{figure}

\section*{Acknowledgments}

The work is supported by the National Natural Science Foundation of China (Grant Nos. 12261097, 12261103), and the Yunnan Fundamental Research Projects (Grant Nos. 202101AT070208, 202301AT070117, 202101AS070044) and Innovation team of School of Mathematics and Statistics, Yunnan University (No.
ST20210104). The first author is in particular grateful to Prof. Weizhu Bao for fruitful discussions.

\section*{Conflict of interest} The authors declare that they have no conflict of interest.




\end{document}